\documentclass[11pt,reqno]{amsart}
\usepackage[utf8]{inputenc}
\usepackage{lmodern} 
\usepackage{amsmath,amssymb,amsfonts,amsthm,latexsym}
\usepackage{graphicx,multirow,enumerate}
\usepackage{tikz}
\usetikzlibrary{calc}
\usepackage{booktabs} 
\usepackage{float}
\usepackage{caption}
\usepackage{array}      
\usepackage{cellspace}  
\usepackage{mathrsfs}  
\usepackage[all]{xy}   
\usepackage{color,xcolor} 
\usepackage{cancel}     
\usepackage{url}        
\usepackage{longtable}
\usepackage{hyperref}
\theoremstyle{plain}
\newtheorem{thm}{Theorem}[section]
\newtheorem{lem}[thm]{Lemma}
\newtheorem{prop}[thm]{Proposition}
\newtheorem{cor}[thm]{Corollary}
\theoremstyle{definition}
\newtheorem{ex}[thm]{Example}
\newtheorem{defn}[thm]{Definition}
\newtheorem{rem}[thm]{Remark}

\begin{document}
\title[]{Gelfand--Dorfman Algebras: Nilpotency, Solvability, Construction and Classification}

    \author{Ziyi Zhang}
	\address{School of Mathematics and Statistics, Northeast Normal University, Changchun 130024, Jilin, China}
	\email{zhangziyi@nenu.edu.cn}
	
	\author{Zeyu Hao}
	\address{School of Mathematics and Statistics, Northeast Normal University, Changchun 130024, Jilin, China}
	\email{haozeyu@nenu.edu.cn}
	
	\author{Yining Sun}
	\address{School of Mathematics and Statistics, Northeast Normal University, Changchun 130024, Jilin, China}
	\email{yiningsun@nenu.edu.cn}
	
	\author{Liangyun Chen*}
	\address{School of Mathematics and Statistics, Northeast Normal University, Changchun 130024, Jilin, China}
	\email{chenly640@nenu.edu.cn}

\begin{abstract}
In this paper, we characterize the nilpotency and solvability of Gelfand--Dorfman (GD) algebras. In contrast with Poisson algebras and transposed Poisson algebras, we give examples show the nilpotency and solvability of a GD algebra are not determined by the nilpotency and solvability of its underlying algebras. To obtain more examples of special and non-special GD algebras, we give several construction methods and determine whether the resulting algebras are special. Futhermore, we study GD algebra structures on simple Lie algebras. We provide examples demonstrating that GD algebra structures on simple Lie algebras are not necessarily trivial, distinguishing them from Poisson and transposed Poisson algebras. Finally, we provide a complete algebraic classification of low-dimensional complex GD algebras, and determine their nilpotency, solvability and speciality.
\end{abstract}

	\renewcommand{\thefootnote}{\fnsymbol{footnote}}

\thanks{*Corresponding author.}
\thanks{This work is supported by NNSF of China (No. 12271085).}

\maketitle

\section{Introduction}

Gelfand--Dorfman (GD) algebras, originated from the study of
Hamiltonian operators and Hamiltonian pairs in integrable systems
\cite{Gelfand-Dorfman1979,Dorfman1993}. They are algebraic structures carrying simultaneously a Lie bracket and a Novikov product subject to a compatibility identity. 
The correspondence between quadratic conformal algebras and GD algebra structures originated in the Hamiltonian-pair theory of Gelfand and Dorfman\cite{Gelfand-Dorfman1979}.
Xu established the correspondence between quadratic conformal superalgebras and GD super algebras in \cite{Xu2000} and classified them via compatible Lie superalgebra and Novikov superalgebra structures. He also constructed GD algebras from classical \(R\)-matrices in \cite{Xu2002}.
 A standard construction of GD algebras arises from differential Poisson algebras. A GD algebra is called special if it can be embedded into a differential Poisson algebra. Kolesnikov and his collaborators explored  the speciality of GD algebras. In \cite{Kolesnikov-Sartayev-Orazgaliev}, they exhibited two independent special identities of degree 4. In
\cite{Kolesnikov-Sartayev}, they also proved that all special identities of degree at most 5 are consequences of the two degree 4 identities, while whether these identities generate all special identities remains open. In \cite{Kolesnikov-Panasenko}, they proved that Novikov commutator algebras are special. The relation with Novikov--Poisson algebras was further clarified in \cite{Kolesnikov-Nesterenko}, where it was shown that every commutator GD algebra (known as a transposed Poisson algebra) obtained from a Novikov--Poisson algebra is special. Quadratic Lie conformal superalgebras related to Novikov superalgebras were studied in \cite{KKP2021}, and free special GD algebras were further investigated in \cite{Gubarev-Sartayev}. On the other hand, Hong and collaborators studied quadratic Lie conformal algebras and their simplicity, central extensions and conformal derivations \cite{Hong-Wu,Hong2022}. They developed extending structures for GD algebras \cite{Wen-Hong2024} and Hong, Bai and Guo introduced a bialgebra theory of GD algebras, with applications to Lie conformal bialgebras \cite{Hong-Bai-Guo2025}. 

	Transposed Poisson algebras form a particular class of GD algebras. Transposed Poisson algebras were first introduced in \cite{Bai-Bai-Guo-Wu2023}, where they were defined as a dual notion to Poisson algebras by exchanging the roles of the associative product and the Lie bracket in the Leibniz rule. Research on transposed Poisson algebras has gained significant attention in recent years. Transposed Poisson algebra structures on simple Lie algebras were studied in \cite{FernandezOuaridi2024, FernandezOuaridi2026}, and the classification of complex \(3,4\)-dimensional transposed Poisson algebras was obtained in \cite{BeitesFernandezKaygorodov,Yuan,Zhang,Zhang4D}. Moreover, further progress has been made in the study of transposed Poisson structures on specific algebraic frameworks. By building
the connection between \(\frac12\)-derivations of Lie algebras and transposed Poisson algebras \cite{Ferreira}, descriptions of all transposed Poisson structures on certain Lie algebras have been obtained. These include the Witt algebra, the Virasoro algebra and the \(W(a,b)\) algebra \cite{Ferreira}, Block Lie algebras and superalgebras \cite{Kaygorodov-Block}, solvable Lie algebras with filiform nilradical \cite{Abdurasulov-Adashev-Eshmeteva},
Galilean and solvable Lie algebras \cite{Kaygorodov-Lopatkin-Zhang}, and other classes of Lie algebras. Further developments include Hom and BiHom versions
\cite{LaraiedhSilvestrov2021,MaLi2023}, and the recent study of nilpotency and Frattini theory for transposed Poisson algebras \cite{JinHong2026}.

The first purpose of this paper is to study nilpotency and solvability for GD algebras. These questions are natural in view of recent work on Poisson algebras \cite{F2025}, where these properties are completely determined by the nilpotency and solvability of the underlying algebras. A similar phenomenon also occurs in generalized Poisson algebras, transposed Poisson algebras and poisson \(n\)-lie algebras~\cite{JinHong2026,Cao2025,Cao2026}.
Regarding nilpotency, we show that for GD algebras of dimension at most 3, the nilpotency is completely determined by the underlying structures. However, we provide 4-dimensional examples to demonstrate that this property is not determined by its underlying algebras. Furthermore, we establish a nilpotency criterion and derive an Engel-type theorem. These results recover and extend the main nilpotency theorems for transposed Poisson algebras presented in \cite{JinHong2026}. As for solvability, we prove that for GD algebras of dimension up to 3, their solvability is fully characterized by that of their underlying algebras. However, we construct an explicit counterexample demonstrating that this
characterization fails in dimension \(4\). Furthermore, we provide a characterization of solvability, which we then apply to characterize the solvability of transposed Poisson algebras. Finally, we study the relationship between nilpotency and solvability in GD algebras. We show that for dimensions $n \leq 3$, a GD algebra is solvable if and only if its derived algebra is nilpotent. However, this equivalence is no longer valid in dimension 4. We further characterize the relationship between solvability and the nilpotency of the derived algebra in the general case. Additionally, as an application of the result, we obtain corresponding solvability criteria for transposed Poisson algebras.

To obtain further examples of special and non-special GD algebras, we present several constructions of GD algebras. Some constructions produce special GD algebras, whereas others give non-special examples detected by the degree-4 special identities. We also study GD algebra structures on simple Lie algebras. For a fixed Lie algebra, the GD compatibility condition can be
interpreted as a Chevalley--Eilenberg \(1\)-cocycle condition for the right multiplication map. We classify all GD products on the Lie algebra $\mathfrak{sl}_2(\mathbb{C})$ and show that it admits non-trivial GD algebra structures. Unlike Poisson and transposed Poisson algebras, which are known to be trivial on simple Lie algebras \cite{FernandezOuaridi2024, BenayadiBoucetta2014}. Finally, we give low-dimensional classification results over \(\mathbb C\), including the classification of complex 2,3-dimensional GD algebras and the determination of their nilpotency, solvability and speciality.

The paper is organized as follows. Section~\ref{sec:preliminaries} contains the basic definitions and preliminary results. Section~\ref{sec:nilpotency} studies nilpotency of GD algebras. We prove that, in dimensions at most \(3\), a GD algebra is nilpotent if and only if both its underlying Lie algebra and its underlying Novikov algebra are nilpotent. We construct a \(4\)-dimensional counterexample showing that this characterization fails in general. We then give a nilpotency criterion and derive an
Engel-type theorem, and obtain the nilpotency characterization for transposed Poisson algebras as a consequence.
Section~\ref{sec:solvability} is devoted to solvability. We prove that, in dimensions at most \(3\), the solvability of a GD algebra is completely determined by the solvability of its underlying algebras, whereas a \(4\)-dimensional counterexample shows that this is no longer true in general. We further characterize the solvability of arbitrary GD algebras and obtain the corresponding solvability criterion for transposed Poisson algebras. Finally, we study the relationship between nilpotency and solvability. We show that the solvability of a GD algebra is equivalent to the nilpotency of its derived algebra in dimensions up to 3, but provide counterexamples in dimension 4. A general characterization is then established and applied to derive solvability criteria for transposed Poisson algebras. Section~\ref{sec:constructions}
presents several construction methods for GD algebras and determines whether the resulting algebras are special or non-special. Section~\ref{sec:simple-lie} studies GD algebra structures on simple Lie algebras, in particular, we classify all GD products on \(\mathfrak{sl}_2(\mathbb C)\). Finally,
Section~\ref{sec:classification-3d} gives the classification of complex 2,3-dimensional GD algebras, together with the determination of their nilpotency, solvability and speciality.

Throughout this paper, unless otherwise stated, all vector spaces and algebras are assumed to be finite-dimensional over an algebraically closed field \(\Bbbk\) of characteristic zero. The classification results in Section~\ref{sec:classification-3d} are established over the field of complex numbers \(\mathbb C\).

\section{Preliminaries}
\label{sec:preliminaries}
In this section, we provided definitions and preliminary results that will be used throughout of the paper.

\begin{defn}\cite{Xu2000}
A \emph{Gelfand--Dorfman algebra}, or simply a \emph{GD algebra}, is a triple \((A,\circ,[\cdot,\cdot])\) equipped with two bilinear operations such that \((A,\circ)\) is a Novikov algebra, \((A,[\cdot,\cdot])\) is a Lie algebra, and the following compatibility identity holds for all \(x,y,z\in A\):
\begin{equation}
\label{eq:GD_compatibility}
    [x, y \circ z] - [z, y \circ x]
    + [y, x] \circ z - [y, z] \circ x
    - y \circ [x, z] = 0 .
\end{equation}
\end{defn}

\begin{defn}\cite{Kolesnikov-Sartayev-Orazgaliev}
A GD algebra \((A,\circ,[\cdot,\cdot])\) is said to be special if it can be embedded into a GD algebra induced by a differential Poisson algebra. That is, there exists a differential Poisson algebra \((P,\cdot,\{\cdot,\cdot\},d)\) such that \(A\) is isomorphic to a GD subalgebra of \(P^{(d)}\), where \(P^{(d)}\) is the GD algebra on \(P\) defined by
$
        [u,v]=\{u,v\}$ and $
        u\circ v=u\cdot d(v)
      $ for all $ u,v\in P.
$
\end{defn}

\begin{rem}
The class of special GD algebras forms a subvariety of the variety of all GD algebras. An identity which holds in all special GD algebras but does not hold in the whole variety of GD algebras is called a \emph{special identity}. The study of special identities for GD algebras was initiated in \cite{Kolesnikov-Sartayev-Orazgaliev} and further developed in \cite{Kolesnikov-Sartayev}. In particular, \cite{Kolesnikov-Sartayev-Orazgaliev} proved that there are no special identities of degree \(<4\) and exhibited the following two independent special identities of degree 4:
\begin{equation}
\label{eq:GD_identity_35}
    [x_1, x_3 \circ x_2] \circ x_4
    + \bigl([x_3, x_1] \circ x_2\bigr) \circ x_4
    =
    [x_1, (x_3 \circ x_4) \circ x_2]
    + [x_3 \circ x_4, x_1] \circ x_2 ,
\end{equation}
and
\begin{equation}
\label{eq:GD_identity_40}
\begin{aligned}
    2[x_4, (x_3 \circ x_2) \circ x_1]
    &=
    [x_4 \circ x_1, x_3 \circ x_2]
    + [x_4 \circ x_2, x_3 \circ x_1]  \\
    &\quad
    + [x_4, x_3 \circ x_2] \circ x_1
    + [x_4, x_3 \circ x_1] \circ x_2  \\
    &\quad
    + [x_3, x_4 \circ x_1] \circ x_2
    + [x_3, x_4 \circ x_2] \circ x_1 .
\end{aligned}
\end{equation}
Subsequently, \cite{Kolesnikov-Sartayev} proved that every special identity of degree at most 5 is a consequence of these two identities. It remains an open problem whether these two identities generate all special identities. Finally, not every GD algebra is special, non-special examples were found implicitly in \cite{Kolesnikov-Sartayev-Orazgaliev}, and explicit counterexamples were later given in \cite{Kolesnikov-Panasenko,KKP2021}.
\end{rem}

\begin{defn}
Let \((A,\circ,[\cdot,\cdot])\) be a GD algebra. The GD derived series of subspaces for \(n\geq 0\) is given by
\begin{equation}
\label{eq:derived-series}
        A^{(0)} = A,
        \quad
        A^{(n+1)}
        =
        A^{(n)}\circ A^{(n)}
        +
        [A^{(n)},A^{(n)}].
\end{equation}
Likewise, the lower central series is the sequence of subspaces for \(n\geq 1\) given by
\begin{equation}
\label{eq:GD-powers}
        A^{\langle 1\rangle} =A,
        \quad
        A^{\langle n+1\rangle}
        =
        \sum_{i=1}^{n}
        \left(
        A^{\langle i\rangle}\circ A^{\langle n+1-i\rangle}
        +
        [A^{\langle i\rangle},A^{\langle n+1-i\rangle}]
        \right).
\end{equation}
A GD algebra \(A\) is called solvable (nilpotent) if there exists \(N\geq 0\) (\(N\geq 1\)) such that
$
        A^{(N)}=0$ $
        \bigl ( A^{\langle N\rangle}=0\bigr).
$
\end{defn}

\begin{defn}\cite{Cao2025}
A \emph{generalized Poisson algebra} is a tuple
\((A,\cdot,[\cdot,\cdot],D)\), where \((A,\cdot)\) is a commutative associative algebra, \((A,[\cdot,\cdot])\) is a Lie algebra, and \(D\) is a derivation of both \((A,\cdot)\) and \((A,[\cdot,\cdot])\). Moreover, for all \(x,y,z\in A\),
\begin{equation}
\label{eq:Gen_Poisson_compatibility}
        [x\cdot y,z]
        =
        x\cdot [y,z]+[x,z]\cdot y
        +x\cdot y\cdot D(z).
\end{equation}
\end{defn}

\section{Nilpotency of GD Algebras}
\label{sec:nilpotency}
For Poisson algebras~\cite{F2025}, recent results show that nilpotency is equivalent to the simultaneous nilpotency of the underlying Lie algebra and the underlying commutative associative algebra. Related nilpotency questions have also been studied for transposed Poisson algebras, generalized Poisson algebras and Poisson \(n\)-Lie algebras~\cite{JinHong2026,Cao2025,Cao2026}. In this section, we study the nilpotency of GD algebras.

For \(x\in A\), let \(L_x,R_x,\operatorname{ad}_x\in \operatorname{End}_{\Bbbk}(A)\) be given by
\[
        L_x(a)=x\circ a,\quad
        R_x(a)=a\circ x,\quad
        \operatorname{ad}_x(a)=[x,a],
        \quad \text{for }a \in A.
\]
For a subset \(S\) of an associative algebra \(B\), denote by \(\operatorname{alg}_{B}(S)\) the non-unital associative subalgebra of \(B\) generated by \(S\). The multiplication algebra of \(A\) is
$
        \mathcal M(A)
        :=
        \operatorname{alg}_{\operatorname{End}_{\Bbbk}(A)}
        \{L_x,R_x,ad_x\mid x\in A\}.
$

\begin{lem}\label{lem:dial-vs-M}
For a GD algebra \((A,\circ,[\cdot,\cdot])\), the following conditions are equivalent:
\begin{enumerate}
\item \(A\) is nilpotent, that is, \(A^{\langle N\rangle}=0\) for some \(N\geq 1\).
\item all mixed nonassociative monomials in the operations \(\circ\) and \([\cdot,\cdot]\) of sufficiently large length vanish identically on \(A\).
\item \(\mathcal M(A)\) is a nilpotent associative algebra.
\end{enumerate}
\end{lem}

\begin{proof}
Let \(P_n(A)\) be the linear span of all values on \(A\) of mixed nonassociative monomials of length \(n\) in the operations \(\circ\) and \([\cdot,\cdot]\), where the length means the number of variables. By induction on \(n\), using \eqref{eq:GD-powers}, we have
$
        P_n(A)=A^{\langle n\rangle}
        \text{ for }n\geq 1.
$

The sequence \(A^{\langle n\rangle}\) is descending. Indeed, \(A^{\langle 2\rangle}\subseteq A^{\langle 1\rangle}\). If \(A^{\langle j\rangle}\subseteq A^{\langle j-1\rangle}\) for \(2\leq j\leq n\), then each summand in \(A^{\langle n+1\rangle}\) is contained in a summand appearing in the definition of \(A^{\langle n\rangle}\). Hence \(A^{\langle n+1\rangle}\subseteq A^{\langle n\rangle}\). Therefore \(A^{\langle N\rangle}=0\) for some \(N\) is equivalent to the vanishing of all mixed nonassociative monomials of sufficiently large length. This proves the equivalence of \((1)\) and \((2)\).

Assume that all mixed nonassociative monomials of length at least \(N\) vanish identically on \(A\). If \(A=0\), then \(\mathcal M(A)=0\). Otherwise \(N\geq 2\).
Let \(W\) be a word of length at least \(N-1\) in the operators \(L_x,R_x,ad_x\),\text{ where} \(x\in A\). Then \(W(a)\), \text{ for} \(a\in A\), is the value of a mixed nonassociative monomial of length at least \(N\). Hence \(W(a)=0\) for all \(a\in A\), and therefore \(W=0\). Since \(\mathcal M(A)^{N-1}\) is spanned by such operator words, we obtain \(\mathcal M(A)^{N-1}=0\). Thus \(\mathcal M(A)\) is nilpotent.

Conversely, suppose that \(\mathcal M(A)^r=0\). Let \(w\) be a mixed nonassociative monomial of length \(m>2^{r-1}\). In the binary tree corresponding to \(w\), there is a leaf of depth at least \(r\). Fix an evaluation of the variables in \(A\), and follow the path from this leaf to the root. Along this path, each step acts on the current value by one of the operators \(L_b,R_b,ad_b\), or by \(-ad_b\), where \(b\in A\) is the value of the corresponding sibling submonomial. Thus the value of \(w\) can be written as \(T(a)\), where \(a\in A\) and \(T\) is an operator word of length at least \(r\) in the generators of \(\mathcal M(A)\). Since \(\mathcal M(A)^r=0\), it follows that \(T=0\). Therefore every mixed nonassociative monomial of length
\(m>2^{r-1}\) vanishes identically on \(A\), proving \((2)\).
\end{proof}

\begin{lem}\label{lem:basic-identities}
Let \((A,\circ,[\cdot,\cdot])\) be a GD algebra. Then, for all \(x,y,z\in A\),
\begin{align}
        R_zL_x &= L_{x\circ z}, \label{eq:RL-collapse}\\
        \operatorname{ad}_xL_y
        &= L_y\operatorname{ad}_x - \operatorname{ad}_{y\circ x}
        + L_{[x,y]} + R_x\operatorname{ad}_y .
        \label{eq:QL-transfer}
\end{align}
\end{lem}

\begin{proof}
The result follows by direct computation.
\end{proof}

Let \(\mathcal I_{R\operatorname{ad}}\triangleleft \mathcal M(A)\) be the two-sided ideal generated by all operators \(R_x\operatorname{ad}_y\), $ where $ $x,y\in A$, and set
$
        \overline{\mathcal M}(A)
        :=
        \mathcal M(A)/\mathcal I_{R\operatorname{ad}} .
$

\begin{lem}
\label{lem:normal-form}
The algebra \(\overline{\mathcal M}(A)\) is linearly spanned by ordered products of the form
\[
        \overline L_{x_1}\cdots \overline L_{x_p}\,
        \overline{\operatorname{ad}}_{y_1}\cdots
        \overline{\operatorname{ad}}_{y_q}\,
        \overline R_{z_1}\cdots \overline R_{z_r},
        \quad
        x_i,y_j,z_k\in A,\quad p,q,r\geq 0,\quad p+q+r\geq 1.
\]
\end{lem}

\begin{proof}
In \(\overline{\mathcal M}(A)\), the defining ideal gives
$
        \overline R_x\,\overline{\operatorname{ad}}_y=0
        $ for all $ x,y\in A,
$
while Lemma~\ref{lem:basic-identities} gives
$
        \overline R_z\,\overline L_x=\overline L_{x\circ z},
$
and
$
        \overline{\operatorname{ad}}_x\,\overline L_y
        =
        \overline L_y\,\overline{\operatorname{ad}}_x
        -\overline{\operatorname{ad}}_{y\circ x}
        +\overline L_{[x,y]} .
$
Order the generators by \(L<\operatorname{ad}<R\). Call a product regular if all \(L\)-factors come first, followed by all \(\operatorname{ad}\)-factors and then all \(R\)-factors. The only adjacent pairs which violate this order are \(RL\), \(\operatorname{ad}L\), and \(R\operatorname{ad}\). The relations above replace \(RL\) by a single \(L\)-factor, replace \(R\operatorname{ad}\) by zero, and replace \(\operatorname{ad}L\) by a sum in which the only term of the same length has strictly smaller inversion number, while all other terms have smaller length. Therefore induction on the lexicographically ordered pair
consisting of the length and the inversion number reduces every product to a linear combination of regular products. 
\end{proof}

\begin{thm}\label{thm:main}
Let \((A,\circ,[\cdot,\cdot])\) be a finite-dimensional GD algebra. Then \(A\) is nilpotent if and only if the following three conditions hold:
\begin{enumerate}
\item the Lie algebra \((A,[\cdot,\cdot])\) is nilpotent.
\item the Novikov algebra \((A,\circ)\) is nilpotent.
\item the two-sided ideal
$
        \mathcal I_{R\operatorname{ad}}
$
is nilpotent.
\end{enumerate}
\end{thm}

\begin{proof}
By Lemma~\ref{lem:dial-vs-M}, \(A\) is nilpotent if and only if
\(\mathcal M(A)\) is nilpotent. Hence necessity is immediate: nilpotency of \(\mathcal M(A)\) forces the nilpotency of the associative subalgebras generated by the \(\operatorname{ad}_x\)'s and by the \(L_x,R_x\)'s, and therefore the nilpotency of the
underlying Lie and Novikov algebras. It also implies the nilpotency of every two-sided ideal, in particular of \(\mathcal I_{R\operatorname{ad}}\).

Conversely, assume \((1)\)--\((3)\), and set
$
        \overline{\mathcal M}:=\mathcal M(A)/\mathcal I_{R\operatorname{ad}}.
$
Since \(\mathcal I_{R\operatorname{ad}}\) is nilpotent, it suffices to prove that \(\overline{\mathcal M}\) is nilpotent. Indeed, if \(\overline{\mathcal M}^{\,s}=0\), then \(\mathcal M(A)^s\subseteq \mathcal I_{R\operatorname{ad}}\), and hence \(\mathcal M(A)^{st}=0\) for any \(t\) such that \(\mathcal I_{R\operatorname{ad}}^{\,t}=0\).
Let
\[
        \mathcal A_L
        :=
        \operatorname{alg}_{\overline{\mathcal M}}
        \{\overline L_x\mid x\in A\},\quad
        \mathcal A_{\operatorname{ad}}
        :=
        \operatorname{alg}_{\overline{\mathcal M}}
        \{\overline{\operatorname{ad}}_x\mid x\in A\},\quad
        \mathcal A_R
        :=
        \operatorname{alg}_{\overline{\mathcal M}}
        \{\overline R_x\mid x\in A\}.
\]
By \((1)\), \(\mathcal A_{\operatorname{ad}}\) is nilpotent, since it is generated by the adjoint operators of the nilpotent Lie algebra \((A,[\cdot,\cdot])\). By \((2)\), the algebras \(\mathcal A_L\) and \(\mathcal A_R\) are nilpotent.

Suppose that \(\overline{\mathcal M}\) is not nilpotent. Since it is finite-dimensional, the standard simple-module criterion for finite-dimensional non-unital associative algebras yields a nonzero simple left \(\overline{\mathcal M}\)-module \(S\) such that \(\overline{\mathcal M}S=S\). Since \(\mathcal A_R\) is nilpotent, there exists \(0\neq w_1\in S\) with \(\mathcal A_Rw_1=0\). Moreover,
\(\mathcal A_R\mathcal A_{\operatorname{ad}}=0\) in \(\overline{\mathcal M}\), because
\(\overline R_x\overline{\operatorname{ad}}_y=0\) for all \(x,y\in A\). Hence \(S_R:=\{s\in S\mid \mathcal A_Rs=0\}\) is stable under \(\mathcal A_{\operatorname{ad}}\). Since \(\mathcal A_{\operatorname{ad}}\) is nilpotent, there exists \(0\neq w\in S_R\) such that
$
        \mathcal A_Rw=0$ and $ \mathcal A_{\operatorname{ad}}w=0.
$

The submodule \(\overline{\mathcal M}w\) is nonzero, otherwise \(\Bbbk w\) would be a nonzero submodule annihilated by \(\overline{\mathcal M}\), contrary to \(\overline{\mathcal M}S=S\). Thus \(S=\overline{\mathcal M}w\). By Lemma~\ref{lem:normal-form}, \(\overline{\mathcal M}\) is linearly spanned by ordered products
$
        \overline L_{x_1}\cdots \overline L_{x_p}
        \overline{\operatorname{ad}}_{y_1}\cdots \overline{\operatorname{ad}}_{y_q}
        \overline R_{z_1}\cdots \overline R_{z_r}.
$
All products with \(q>0\) or \(r>0\) annihilate \(w\). Hence \(S=\mathcal A_Lw\), so \(w=Uw\) for some \(U\in\mathcal A_L\). Since \(\mathcal A_L\) is nilpotent, \(U^m=0\) for some \(m\), and therefore \(w=U^mw=0\), a contradiction.
Thus \(\overline{\mathcal M}\) is nilpotent. Since \(\mathcal I_{R\operatorname{ad}}\) is nilpotent, \(\mathcal M(A)\) is nilpotent. Lemma~\ref{lem:dial-vs-M} then implies that \(A\) is nilpotent.
\end{proof}

\begin{rem}
Theorem~\ref{thm:main} shows that, besides the nilpotency of the
underlying Lie and Novikov algebras, one must also control the two-sided ideal \(\mathcal I_{R\operatorname{ad}}\). Thus the nilpotency of a GD algebra is not determined by the nilpotency of its underlying Lie and Novikov algebras. The next two propositions show that this obstruction does not occur in dimensions at most \(3\), while it already appears in dimension \(4\).
\end{rem}

\begin{prop}\label{prop:no-low-dimensional-counterexample}
Let \((A,\circ,[\cdot,\cdot])\) be a finite-dimensional GD algebra with \(\dim A\leq 3\). If the Lie algebra \((A,[\cdot,\cdot])\) and the Novikov algebra \((A,\circ)\) are both nilpotent, then \(A\) is nilpotent.
\end{prop}

\begin{proof}
By Theorem~\ref{thm:main}, it is enough to prove that \(\mathcal I_{R\operatorname{ad}}\) is nilpotent. If \(\dim A\leq 2\), then the nilpotent Lie algebra \((A,[\cdot,\cdot])\) is abelian. Hence \(\operatorname{ad}_x=0\) for all \(x\in A\), and so
\(\mathcal I_{R\operatorname{ad}}=0\).

Assume now that \(\dim A=3\). If the Lie algebra is abelian, the same argument applies. Otherwise it is the three-dimensional Heisenberg Lie algebra. Choose a basis \(\{e_1,e_2,c\}\) such that \([e_1,e_2]=c\), and put \(C:=\Bbbk c\). Then \(C\) is the center and \([A,A]=C\), so \(\operatorname{ad}_x(A)\subseteq C\) and \(\operatorname{ad}_x(C)=0\) for all \(x\in A\).

We claim that \(C\) is stable under all left and right Novikov multiplications. Putting \(y=c,x=e_1,z=e_2\) in the GD compatibility identity gives \(c\circ c\in C\). Since \((A,\circ)\) is nilpotent, this implies \(c\circ c=0\).
Taking \(z=c\) in the compatibility identity gives \([x,y\circ c]+[y,x]\circ c=0\). Since \([y,x]\in C\) and \(c\circ c=0\), we get \([x,y\circ c]=0\) for all \(x,y\in A\). Hence \(A\circ C\subseteq C\). Finally, taking \(x=e_1\) and \(z=e_2\) and reducing modulo \(C\), we get
$
        [y,e_1]\circ e_2-[y,e_2]\circ e_1\equiv 0 \pmod C.
$
Putting \(y=e_1\) and \(y=e_2\), respectively, gives \(c\circ e_1,c\circ e_2\in C\). Together with \(c\circ c=0\), this proves \(C\circ A\subseteq C\).

Thus every operator \(R_z\operatorname{ad}_x\) maps \(A\) into \(C\) and annihilates \(C\). This property is preserved under multiplication on both sides by the generators \(L_a,R_a,\operatorname{ad}_a\) of \(\mathcal M(A)\). Hence every element of \(\mathcal I_{R\operatorname{ad}}\) maps \(A\) into \(C\) and annihilates \(C\), and therefore
$
        \mathcal I_{R\operatorname{ad}}^2=0.
$
Thus \(\mathcal I_{R\operatorname{ad}}\) is nilpotent. Theorem~\ref{thm:main} implies that \(A\) is nilpotent.
\end{proof}

\begin{rem}
The proof of the proposition~\ref{prop:no-low-dimensional-counterexample} can also be obtained from the classification of low-dimensional GD algebras presented later in this paper.
\end{rem}

\begin{ex}
\label{ex:four-dimensional-nilpotent-counterexample}
The following example shows that Proposition~\ref{prop:no-low-dimensional-counterexample} cannot be extended to dimension \(4\).
Let
$
        A=\Bbbk e_1\oplus \Bbbk e_2\oplus \Bbbk e_3\oplus \Bbbk e_4.
$
Define a skew-symmetric bilinear bracket by
$
        [e_1,e_2]=e_3,\quad [e_1,e_3]=e_4.
$
 Define a bilinear product by
\[
        e_3\circ e_1=e_2,\quad
        e_1\circ e_3=-e_2,\quad
        e_4\circ e_1=-e_3.
\]
Then \((A,\circ,[\cdot,\cdot])\) is a GD algebra whose underlying Lie algebra and underlying Novikov algebra are both nilpotent, but which is not nilpotent as a GD algebra.

\end{ex}

Theorem~\ref{thm:main} immediately yields an Engel-type criterion.

\begin{thm}
\label{thm:engel}
Let \((A,\circ,[\cdot,\cdot])\) be a GD algebra. Assume that
\begin{enumerate}
\item for every \(x\in A\), the operator \(\operatorname{ad}_x\) is nilpotent.
\item for every \(x\in A\), the operators \(L_x\) and \(R_x\) are nilpotent.
\item the ideal
$
        \mathcal I_{R\operatorname{ad}}
$
is a nil ideal, that is, every element of \(\mathcal I_{R\operatorname{ad}}\) is nilpotent.
\end{enumerate}
Then \(A\) is nilpotent.
\end{thm}

\begin{proof}
By Engel's theorem for Lie algebras, condition \((1)\) implies that the Lie algebra \((A,[\cdot,\cdot])\) is nilpotent. By the Engel theorem for Novikov algebras, condition \((2)\) implies that the Novikov algebra \((A,\circ)\) is nilpotent.

By the standing assumption, \(\mathcal M(A)\) is finite-dimensional. Hence \(\mathcal I_{R\operatorname{ad}}\) is a finite-dimensional nil ideal of an associative algebra, and therefore it is nilpotent by Levitzki's theorem. Thus all three conditions of Theorem~\ref{thm:main} are satisfied. Consequently, the GD algebra \(A\) is nilpotent.
\end{proof}

The following corollary recovers the nilpotency criterion for transposed Poisson algebras obtained in \cite{JinHong2026}.

\begin{cor}\label{cor:TP-nilpotent-criterion}
Let \((P,\cdot,[\cdot,\cdot])\) be a transposed Poisson algebra. Then \(P\) is nilpotent if and only if the underlying commutative associative algebra \(P_A\) and the underlying Lie algebra \(P_L\) are nilpotent.
\end{cor}

\begin{proof}
Regard \(P\) as a GD algebra by setting \(x\circ y=x\cdot y\). If \(P\) is nilpotent, then the nilpotency of \(P_A\) and \(P_L\) follows immediately from the definition of the GD lower central series.

Conversely, assume that \(P_A\) and \(P_L\) are nilpotent. Then every \(\operatorname{ad}_x\) is nilpotent, and since \(L_x=R_x\) is the multiplication by \(x\), every \(L_x\) and \(R_x\) is nilpotent. Moreover, the transposed Poisson identity implies
$
        [P_A^i,P_A^j]\subseteq P_A^{i+j-1}
$
for all \(i,j\geq 1\). Hence each \(\operatorname{ad}_x\) preserves the associative-power filtration, while each \(R_x=L_x\) raises it by one. It follows that every element of \(\mathcal I_{R\operatorname{ad}}\) raises this filtration by at least one. Since \(P_A\) is nilpotent, \(\mathcal I_{R\operatorname{ad}}\) is nilpotent, and hence it is a nil ideal.
Therefore Theorem~\ref{thm:engel} applies and \(P\) is nilpotent.
\end{proof}

\section{Solvability of GD Algebras}
\label{sec:solvability}

For both Poisson algebras, generalized Poisson algebras and Poisson \(n\)-Lie algebras~\cite{Cao2025,F2025,Cao2026}, the solvability of the algebra is equivalent to the simultaneous solvability of its underlying algebra. In this section, we study the solvability of GD algebras.

\begin{prop}
\label{thm:dim3}
Let \((A,\circ,[\cdot,\cdot])\) be a GD algebra with \(\dim A\leq 3\). If the underlying Lie algebra \((A,[\cdot,\cdot])\) and the underlying Novikov algebra \((A,\circ)\) are both solvable, then \(A\) is solvable.
\end{prop}

\begin{proof}
We first present the 2-dimensional case. If the Lie bracket is zero, then the GD derived series coincides with the Novikov derived series. Otherwise \(\dim A=2\), and we may choose a basis \(p,q\) such that \([p,q]=q\). The GD compatibility identity forces the product to have the form
\[
        p\circ p=\alpha p+\beta q,\quad
        p\circ q=-\gamma p+\delta q,\quad
        q\circ p=\gamma p+\alpha q,\quad
        q\circ q=\gamma q .
\]
The Novikov identities imply \(2\gamma^2=0\), hence \(\gamma=0\). If \(\alpha\neq 0\), then repeated squaring gives nonzero elements with nonzero \(p\)-component in all terms of the Novikov derived series, contradicting the solvability of \((A,\circ)\). Hence \(\alpha=0\). Thus \(A\circ A+[A,A]\subseteq \Bbbk q\), and since \(q\circ q=[q,q]=0\), we get \(A^{(2)}=0\). Therefore the assertion holds for \(\dim A\leq 2\).

Assume now that \(\dim A=3\) and set \(D:=A\circ A+[A,A]\). If \(D\neq A\), then \(\dim D\leq 2\) and \(D\) is a GD subalgebra since \(D\circ D\subseteq A\circ A\subseteq D\) and \([D,D]\subseteq [A,A]\subseteq D\). The induced Lie and Novikov algebras on \(D\) are solvable, so the 2-dimensional case implies that \(D\) is solvable. As \(A^{(1)}=D\), the GD algebra \(A\) is solvable. Hence it remains only to exclude the case
\begin{equation}
\label{eq:dim3-D=A}
        A\circ A+[A,A]=A.
\end{equation}

Let \(L:=[A,A]\). Since the Lie algebra is 3-dimensional and solvable, \(\dim L\in\{0,1,2\}\). If \(L=0\), then \eqref{eq:dim3-D=A} gives \(A\circ A=A\), contradicting the solvability of \((A,\circ)\).

Assume \(\dim L=2\). If \(L\) were non-abelian, then every derivation of \(L\) would map \(L\) into \([L,L]\). Taking \(a_0\in A\setminus L\), the derivation \(\operatorname{ad}_{a_0}|_L\) would give
\([A,A]=[a_0,L]+[L,L]\subseteq [L,L]\), a contradiction. Hence \(L\) is abelian. Write \(A=\Bbbk a\oplus L\). Then \(\operatorname{ad}_a|_L\) is invertible. Let
\(\eta:A\to\Bbbk\) be the projection onto \(\Bbbk a\). For \(u,v,w\in L\), the GD identity gives
$
        \eta(v\circ w)[a,u]=\eta(v\circ u)[a,w],
$
and therefore \(L\circ L\subseteq L\). Set \(C(u):=\eta(a\circ u)\). The GD identity with \(x=a,y=a,z=u\), together with the Novikov identity
\((u\circ a)\circ v=(u\circ v)\circ a\), gives
$
        \eta(u\circ a)=-C(u)$ and $
        C(u\circ v)=C(u)C(v).
$
The solvability of \((A,\circ)\) forces \(C=0\). The same repeated-squaring argument gives \(\eta(a\circ a)=0\). Thus \(A\circ A\subseteq L\), contradicting \eqref{eq:dim3-D=A}.

It remains to consider \(\dim L=1\), say \(L=\Bbbk e\). Suppose first that \(e\) is central. Put \(P:=A\circ A\). Since \((A,\circ)\) is solvable, \(P\neq A\), while \eqref{eq:dim3-D=A} gives \(P+\Bbbk e=A\). Hence \(A=P\oplus\Bbbk e\), and the bracket on \(P\) is given by a non-degenerate alternating form \([u,v]=\omega(u,v)e\). Applying the GD identity to
\((x,y,z)=(e,u,v)\) and \((u,v,v)\), with \(u,v\in P\), yields
\[
        e\circ e=0,\quad P\circ e=0,\quad e\circ P=0.
\]
Thus all products involving \(e\) vanish, and so \(A\circ A=P\circ P=P\), again contradicting Novikov solvability.

Finally suppose that \(L\) is not central. Then there is a basis \(h,e,c\) of \(A\) such that \([h,e]=e\) and \(c\) is central. A direct substitution into the GD identity gives
\[
\begin{gathered}
        h\circ h=a h+b e+d c,\quad
        h\circ e=-t h+p e-\beta c,\quad
        e\circ h=t h+a e+\beta c,\quad
        e\circ e=t e,\\
        h\circ c=s h+r c,\quad
        e\circ c=s e,\quad
        c\circ h=q e+m c,\quad
        c\circ e=n e,\quad
        c\circ c=\ell c .
\end{gathered}
\]
The Novikov identities imply
$
        2t^2=0, \beta s=0$ and $ s(s-\ell)=0.
$
Thus \(t=0\). If \(s\neq 0\), then \(\beta=0\) and \(\ell=s\). In this case \(\Bbbk e\) is a Novikov ideal, and in \(A/\Bbbk e\) we have
$
        \bar h\circ \bar c=s\bar h+r\bar c$ and $
        \bar c\circ \bar c=s\bar c,
$
so \((A/\Bbbk e)\circ(A/\Bbbk e)=A/\Bbbk e\), contradicting Novikov
solvability. Hence \(s=0\). With \(t=s=0\), the only product with a possible \(h\)-component is \(h\circ h=a h+b e+d c\). If \(a\neq 0\), repeated squaring contradicts Novikov solvability. Therefore \(a=0\), and hence
$
        A\circ A+[A,A]\subseteq \operatorname{span}\{e,c\}\neq A,
$
contradicting \eqref{eq:dim3-D=A}. All cases lead to contradictions. Therefore \(A\) is solvable.
\end{proof}
\begin{rem}
    The proof of the above proposition can be derived from the classification of low-dimensional GD algebras presented later in this paper.
\end{rem}

\begin{ex}
\label{ex:four-dimensional-nonsolvable-solvable-underlying}
The following example shows that Theorem~\ref{thm:dim3} cannot be extended to dimension \(4\). Let \(A=\operatorname{span}\{h,c,e,f\}\). Define a skew-symmetric bilinear bracket by
\[
        [h,e]=e,\quad [h,f]=-f,\quad [e,f]=c,
\]
where \(c\) is central.
Define a bilinear product \(\circ\) by
\[
        c\circ e=e,\quad c\circ f=-f,\quad e\circ f=-h,\quad f\circ e=h.
\]
Then \((A,\circ,[\cdot,\cdot])\) is a GD algebra. The underlying Lie algebra and the underlying Novikov algebra are solvable. However, \((A,\circ,[\cdot,\cdot])\) is not solvable as a GD algebra.
\end{ex}

Next, we establish a characterization of the solvability of GD algebras. For a GD algebra \((A,\circ,[\cdot,\cdot])\), define \(\Sigma(A):=[A,A]+\delta(A,A)\), where \(\delta(x,y):=x\circ y-y\circ x\).

The following lemma is the main structural point.

\begin{lem}
\label{lem:Sigma-ideal}
For every GD algebra \(A\), the subspace \(\Sigma(A)\) is a GD ideal of \(A\).
\end{lem}

\begin{proof}
For brevity, set \(xy=x\circ y\), and put \(\Delta:=\delta(A,A)\). The Novikov identities imply
$
        2\delta(x,y)z=\delta(xz,y)+\delta(x,yz)
      $ and $
        z\delta(x,y)=-\delta(x,y)z+\delta(x,zy)-\delta(y,zx).
$
Hence, since the ground field has characteristic zero, \(\Delta A+A\Delta \subseteq \Delta\).

It remains to control the products involving \([A,A]\). For \(x,y,z\in A\), set \(a=[x,y]z\), \(b=[y,z]x\) and \(c=[z,x]y\). Modulo \(\Sigma(A)\), the GD identity gives \(a+b-c\equiv 0\). Cyclically permuting \(x,y,z\), we also get \(b+c-a\equiv 0\) and \(c+a-b\equiv 0\). Thus \(2a,2b,2c\in \Sigma(A)\), and therefore
$
        [A,A]A\subseteq \Sigma(A).
$
Since \(uv-vu=\delta(u,v)\in\Delta\), it follows also that \(A[A,A]\subseteq\Sigma(A)\). Hence
$
        A\Sigma(A)+\Sigma(A)A\subseteq\Sigma(A).
$

Finally, the Jacobi identity gives \([A,[A,A]]\subseteq [A,A]\) and \([A,\Delta]\subseteq [A,A]\) because \(\Delta\subseteq A\). Therefore \([A,\Sigma(A)]\subseteq\Sigma(A)\). Thus \(\Sigma(A)\) is closed under both Novikov multiplications by elements of \(A\) and under the Lie bracket with \(A\), so it is a GD ideal.
\end{proof}

Define the skew-derived tower of a GD algebra \((A,\circ,[\cdot,\cdot])\) by
\[
        \mathcal S_0(A):=A,\quad
        \mathcal S_{n+1}(A):=
        \Sigma(\mathcal S_n(A))
        =
        [\mathcal S_n(A),\mathcal S_n(A)]
        +
        \delta(\mathcal S_n(A),\mathcal S_n(A)).
\]

\begin{thm}
\label{thm:skew-derived}
Let \((A,\circ,[\cdot,\cdot])\) be a GD algebra. Then \(A\) is solvable if and only if the Novikov algebra \((A,\circ)\) is solvable and \(\mathcal S_m(A)=0\) for some \(m\geq 0\).
\end{thm}

\begin{proof}
Assume first that \(A\) is solvable. Then \((A,\circ)\) is solvable, since \(A_{\mathrm{Nov}}^{(n)}\subseteq A^{(n)}\) for all \(n\).
Moreover,
$
\mathcal S_n(A) \subseteq A^{(n)} \text{for } n \geq 0,
$
because \([U,U]+\delta(U,U)\subseteq [U,U]+U\circ U\) for every subspace \(U\subseteq A\). Hence \(\mathcal S_m(A)=0\) for some \(m\).

Conversely, assume that \((A,\circ)\) is solvable and that
\(\mathcal S_m(A)=0\). By Lemma~\ref{lem:Sigma-ideal}, applied successively, \(\mathcal S_{n+1}(A)\) is a GD ideal of \(\mathcal S_n(A)\). Hence each quotient \(\mathcal S_n(A)/\mathcal S_{n+1}(A)\) is a GD algebra. By definition of \(\mathcal S_{n+1}(A)\), this quotient has zero Lie bracket and commutative Novikov product. Since it is a Novikov subquotient of the solvable Novikov algebra \((A,\circ)\), it is solvable as a Novikov algebra, and therefore solvable as a GD algebra.

Starting from \(\mathcal S_m(A)=0\) and using solvability of extensions, we obtain successively that
$
        \mathcal S_{m-1}(A),\ldots,\mathcal S_0(A)=A
$
are solvable GD algebras. Thus \(A\) is solvable.
\end{proof}

\begin{cor}
\label{cor:TP-solvability}
Let \((P,\cdot,[\cdot,\cdot])\) be a transposed Poisson algebra. \((P,\cdot,[\cdot,\cdot])\) is solvable if and only if \((P,\cdot)\) is solvable and \((P,[\cdot,\cdot])\) is solvable.
\end{cor}

\begin{proof}
A transposed Poisson algebra is a GD algebra whose Novikov product is commutative and associative. Hence \(\delta(P,P)=0\), and the skew-derived tower satisfies
$
        \mathcal S_{n+1}(P)=[\mathcal S_n(P),\mathcal S_n(P)].
$
Thus \(\mathcal S_m(P)=0\) for some \(m\) if and only if the Lie algebra
\((P,[\cdot,\cdot])\) is solvable. By Theorem~\ref{thm:skew-derived}, \(P\) is
solvable if and only if \((P,\cdot)\) is solvable as a Novikov algebra and
\(\mathcal S_m(P)=0\) for some \(m\). Since \((P,\cdot)\) is commutative
associative, its Novikov solvability is equivalent to nilpotency. The result
follows.
\end{proof}

It is well known that the solvability of a Lie algebra is characterized by the nilpotency of its derived algebra. This characterization also holds for Poisson algebras, generalized Poisson algebras, and Poisson $n$-Lie algebras~\cite{Cao2025,F2025,Cao2026}. However, the following example demonstrates that this result does not hold for GD algebras.

\begin{prop}
\label{prop:small-dimensional-derived-nilpotent}
Let \((A,\circ,[\, ,\,])\) be a solvable GD algebra. If \(\dim A\leq 3\), then \(A^{(1)}\) is nilpotent.
\end{prop}

\begin{proof}
Put
$
        D=A^{(1)}=A\circ A+[A,A].
$
Then \(D\) is a GD subalgebra of \(A\). Moreover, $D^{(n)}=A^{(n+1)}$ for any $n\geq 0$, so \(D\) is solvable. If \(A\neq 0\), then \(D\neq A\), otherwise the GD derived series of \(A\) would be constant. Hence \(\dim D\leq 2\).

If \(\dim D\leq 1\), then \(D\) is nilpotent. Indeed, if \(D=\Bbbk e\), then \([e,e]=0\) and \(e\circ e=\lambda e\). Since \(D\) is solvable, we must have \(\lambda=0\), and hence \(D^{\langle 2\rangle}=0\).

It remains to consider the case \(\dim D=2\). Suppose, for a contradiction, that \(D\) is not nilpotent. Since \(D\) is solvable, the subspace
$
        E:=D^{(1)}=D\circ D+[D,D]
$
is a nonzero proper subspace of \(D\). Thus \(\dim E=1\). Choose a basis \(e,f\) of \(D\) such that \(E=\Bbbk e\). Since all products and brackets in \(D\) lie in \(E\), we may write
\[
\begin{gathered}
        e\circ e=0,\quad
        e\circ f=\alpha e,\quad
        f\circ e=\beta e,\quad
        f\circ f=\gamma e,\quad
        [e,f]=\delta e .
\end{gathered}
\]
The equality \(e\circ e=0\) follows from the solvability of \(D\). The left-symmetric identity applied to \((e,f,f)\) gives
$
        (e\circ f)\circ f-e\circ(f\circ f)
        =
        (f\circ e)\circ f-f\circ(e\circ f),
$
and therefore \(\alpha^2e=0\). Then \(\alpha=0\). Thus
\[
\begin{gathered}
        e\circ e=0,\quad e\circ f=0,\quad
        f\circ e=\beta e,\quad f\circ f=\gamma e,\quad
        [e,f]=\delta e .
\end{gathered}
\]
If \((\beta,\delta)=(0,0)\), then the Lie bracket on \(D\) is zero and the only possible nonzero Novikov product is \(f\circ f=\gamma e\). Hence every \(\operatorname{ad}_x\) is zero, every \(L_x\) and \(R_x\) is nilpotent, and \(\mathcal I_{R\operatorname{ad}}=0\). By Theorem~\ref{thm:engel}, \(D\) is nilpotent, a contradiction. Therefore
$
        (\beta,\delta)\neq (0,0).
$

Since \(\dim A\leq 3\) and \(\dim D=2\), we may choose \(x\in A\setminus D\) such that
$
        A=\Bbbk x\oplus \Bbbk e\oplus \Bbbk f.
$
Because \(D==A^{(1)}\), all products and brackets in \(A\) take values in \(D\). Set
\[
\begin{gathered}
        x\circ x=pe+qf,\quad
        x\circ e=re+sf,\quad
        e\circ x=ue+wf,\quad
        x\circ f=me+nf,\\
        f\circ x=he+if,\quad
        [x,e]=\rho e+\sigma f,\quad
        [x,f]=\mu e+\nu f .
\end{gathered}
\]
Let \(\Phi(x,y,z)\) denote the left-hand side of the GD compatibility identity, and let \(\mathrm{LS}\), \(\mathrm{RC}\), and \(\mathrm{Jac}\) denote the left-hand sides of the left-symmetric identity, the right-commutativity identity, and the Jacobi identity, respectively. Comparing the indicated coordinates gives
\[
\begin{aligned}
    \mathrm{LS}(x,f,e)_f &= -\beta s, &\quad
    \Phi(e,x,f)_f &= -\delta s, &\quad
    \mathrm{RC}(e,x,e)_e &= \beta w, \\
    \Phi(x,e,f)_f &= -\delta w, &\quad
    \mathrm{Jac}(x,e,f)_f &= \delta\sigma, &\quad
    \Phi(x,e,e)_e &= -\beta\sigma-\delta w.
\end{aligned}
\]
Since \((\beta,\delta)\neq(0,0)\), these equalities imply
$
        s=w=\sigma=0.
$
Next,
\[
\begin{aligned}
        \mathrm{RC}(x,e,f)_e&=-\beta n+\gamma s,&\quad
        \mathrm{Jac}(x,e,f)_e&=-\delta\nu,&\quad
        \Phi(e,x,f)_e&=-\beta\nu+\gamma\sigma+\delta n .
\end{aligned}
\]
Using \(s=\sigma=0\), we get
\[
        \beta n=0,\quad
        \delta\nu=0,\quad
        -\beta\nu+\delta n=0.
\]
Again, since \((\beta,\delta)\neq(0,0)\), it follows that
$
        n=\nu=0.
$
Furthermore, substituting \(s=w=0\) yields
$
         \mathrm{LS}(x,e,x)_e=-u^2
       \text{ and }
       \mathrm{LS}(x,f,x)_f=-i^2,
$
which implies that \(u=i=0\). Finally, by setting \(s=u=\sigma=0\), we obtain
$
       \mathrm{RC}(x,x,e)_e=\beta q
        \text{ and }
       \Phi(x,x,e)_e=-\delta q.
$
It follows that \(q=0\).

Thus all products and brackets involving \(x\) have zero \(f\)-component. The products and brackets inside \(D\) already lie in \(\Bbbk e\). Therefore
$
        A\circ A+[A,A]\subseteq \Bbbk e,
$
contradicting
$
        D=A^{(1)}=\Bbbk e\oplus \Bbbk f.
$
Hence \(D=A^{(1)}\) must be nilpotent.
\end{proof}

\begin{ex}
\label{ex:four-dimensional-solvable-derived-not-nilpotent}
The bound in Proposition~\ref{prop:small-dimensional-derived-nilpotent} is sharp. Let
$
        A=\Bbbk t\oplus \Bbbk a\oplus \Bbbk c\oplus \Bbbk v.
$
Define a bilinear product by
\[
        t\circ c=a,\quad
        c\circ t=-a,\quad
        c\circ v=-v.
\]
Define a skew-symmetric bracket by
$
        [t,c]=c \text{ and }
        [a,v]=v.
$
Then \((A,\circ,[\, ,\,])\) is a solvable GD algebra, but \(A^{(1)}\) is not nilpotent.
\end{ex}

Next, we establish the relationship between the solvability of a GD algebra and the nilpotency of its derived algebra. As an application of this result, we prove that for transposed Poisson algebras, solvability is equivalent to the nilpotency of the derived algebra.

\begin{prop}
\label{prop:GD-derived-nilpotency-reduction}
Let \((A,\circ,[\, ,\,])\) be a finite-dimensional GD algebra. Let
$
        B=A^{(1)}=A\circ A+[A,A]
$
and \(\mathfrak n(A_{\mathrm{Lie}})\) be the nilradical of the Lie algebra \(A_{\mathrm{Lie}}=(A,[\, ,\,])\). Then \(B\) is nilpotent as a GD algebra if and only if the following conditions hold:
\begin{enumerate}
\item \(A\) is solvable as a GD algebra.
\item \(A\circ A\subseteq \mathfrak n(A_{\mathrm{Lie}})\).
\item \((B,\circ)\) is nilpotent as a Novikov algebra.
\item the ideal
$
        \mathcal I_{R\operatorname{ad}}(B)
$
is nilpotent.
\end{enumerate}
\end{prop}

\begin{proof}
First, \(B\) is a GD ideal of \(A\), since for \(a\in A\) and \(b\in B\), we have \(a\circ b, b\circ a\in A\circ A\subseteq B\) and \([a,b]\in[A,A]\subseteq B\). Assume that \(B\) is GD-nilpotent. By Theorem~\ref{thm:main}, \((B,[\, ,\,])\) is Lie nilpotent, \((B,\circ)\) is Novikov nilpotent, and \(\mathcal I_{R\operatorname{ad}}(B)\) is nilpotent. Since \(A/B\) has zero product and zero bracket, \(A\) is GD-solvable. Moreover, \(B\) is a nilpotent Lie ideal of \(A_{\mathrm{Lie}}\), hence
$
        B\subseteq \mathfrak n(A_{\mathrm{Lie}}).
$
Thus \(A\circ A\subseteq \mathfrak n(A_{\mathrm{Lie}})\), and all four conditions hold.

Conversely, assume that the four conditions hold. Since \(A\) is GD-solvable, the Lie algebra \(A_{\mathrm{Lie}}\) is solvable. Hence, by Cartan's criterion,
$
        [A,A]\subseteq \mathfrak n(A_{\mathrm{Lie}}).
$
Together with \(A\circ A\subseteq \mathfrak n(A_{\mathrm{Lie}})\), this gives
$
        B=A\circ A+[A,A]\subseteq \mathfrak n(A_{\mathrm{Lie}}).
$
Therefore \((B,[\, ,\,])\) is Lie nilpotent. Conditions \((3)\) and \((4)\), together with Theorem~\ref{thm:main} applied to \(B\), imply that \(B\) is GD-nilpotent.
\end{proof}

\begin{lem}
\label{lem:nilpotent-half-derivation}
Let \(\mathfrak g\) be a solvable Lie algebra. Let \(D:\mathfrak g\to\mathfrak g\) be a nilpotent linear map satisfying \(D([x,y]) = \frac12\bigl([Dx,y]+[x,Dy]\bigr)\) for all \(x,y\in\mathfrak g\).
Then
$
        D(\mathfrak g)\subseteq \mathfrak n(\mathfrak g),
$
where \(\mathfrak n(\mathfrak g)\) is the nilradical of \(\mathfrak g\).
\end{lem}

\begin{proof}
Let \(N=[\mathfrak g,\mathfrak g]\). Since \(\mathfrak g\) is solvable, \(N\)
is nilpotent. The defining identity of \(D\) gives \(D(N)\subseteq N\) and by induction, \(D\gamma_r(N)\subseteq\gamma_r(N)\) for \(r\geq 1\),
where \(\gamma_1(N)=N\) and \(\gamma_{r+1}(N)=[N,\gamma_r(N)]\).

Set \(V_r=\gamma_r(N)/\gamma_{r+1}(N)\). The adjoint action of \(\mathfrak g\) on \(V_r\) factors through the abelian Lie algebra \(\mathfrak g/N\), hence admits a generalized weight decomposition
$
        V_r=\bigoplus_{\alpha}V_{r,\alpha}.
$
Let \(\rho_r(x)\) be the induced action of \(x\) on \(V_r\), and let \(D_r\) be induced by \(D\). Then
$
        \rho_r(Dx)=2D_r\rho_r(x)-\rho_r(x)D_r.
$

Let \(T_{\beta\alpha}:V_{r,\alpha}\to V_{r,\beta}\) be the \((\beta,\alpha)\)-block of \(D_r\). For \(\beta\neq\alpha\), the \((\beta,\alpha)\)-block of \(\rho_r(Dx)\) is zero, so
$
        \rho_{r,\beta}(x)T_{\beta\alpha}
        =
        2T_{\beta\alpha}\rho_{r,\alpha}(x).
$
Thus \(T_{\beta\alpha}=0\) unless \(\beta=2\alpha\). Since the set of weights is finite and no nonzero weight can lie in a nontrivial cycle
\(\alpha\mapsto2\alpha\mapsto4\alpha\mapsto\cdots\), the weights can be ordered so that \(D_r\) is block upper triangular. Since \(D_r\) is nilpotent, each diagonal block \(T_{\alpha\alpha}\) is nilpotent.

We show that
$
        \alpha(Dx)=0
        \text{ for every weight }\alpha\text{ and every }x\in\mathfrak g.
$
There is nothing to prove if \(\alpha=0\). Suppose \(\alpha\neq0\), and first take \(x\in\mathfrak g\) with \(\alpha(x)\neq0\). On \(V_{r,\alpha}\), set
\[
        S=\rho_{r,\alpha}(x),\quad
        T=T_{\alpha\alpha},\quad
        C=\rho_{r,\alpha}(Dx).
\]
Then \(C=2TS-ST\). Since the representation factors through the abelian quotient \(\mathfrak g/N\), we have \(CS=SC\). The operator \(S\) has the single nonzero eigenvalue \(\alpha(x)\), and hence is invertible. Put
$
        \Phi(U)=SUS^{-1}.
$
From \(C=2TS-ST\) and \(CS=SC\), we get
$
        (\Phi-I)(\Phi-2I)(T)=0.
$
Since \(\Phi\) has only the eigenvalue \(1\), \(\Phi-2I\) is invertible. Thus \(\Phi(T)=T\), i.e. \(ST=TS\). Hence
$
        C=2TS-ST=TS.
$
As \(T\) is nilpotent and commutes with \(S\), \(C\) is nilpotent. But \(C=\rho_{r,\alpha}(Dx)\) has the single eigenvalue \(\alpha(Dx)\). Hence \(\alpha(Dx)=0\). For arbitrary \(x\), replace \(x\) by \(x+tx_0\), where \(\alpha(x_0)\neq0\), and use linearity in \(t\).

Thus \(\operatorname{ad}_{Dx}\) has only zero eigenvalues on every quotient \(\gamma_r(N)/\gamma_{r+1}(N)\), and it acts trivially on \(\mathfrak g/N\). Hence \(\operatorname{ad}_{Dx}\) is nilpotent on \(\mathfrak g\). Since in a finite-dimensional solvable Lie algebra over an algebraically closed field of characteristic \(0\), the nilradical consists precisely of the ad-nilpotent elements, we get
$
        Dx\in \mathfrak n(\mathfrak g)
      \text{ for all }x\in\mathfrak g.
$
Therefore \(D(\mathfrak g)\subseteq \mathfrak n(\mathfrak g)\).
\end{proof}

\begin{lem}
\label{lem:TP-square-in-nilradical}
Let $(A, \cdot, [\cdot, \cdot])$ be a transposed Poisson algebra. If $A$ is solvable as a GD algebra, then $A \cdot A \subseteq \mathfrak{n}(A_{\mathrm{Lie}})$.
\end{lem}

\begin{proof}
For each $z \in A$, the multiplication operator $L_z(a) = z \cdot a$ satisfies $L_z([x,y]) = \frac{1}{2}([L_zx,y] + [x,L_zy])$ for all $x, y \in A$, by the transposed Poisson identity. Thus \(L_z\) is a \(\frac12\)-derivation of the Lie algebra \(A_{\mathrm{Lie}}\).

Since \(A\) is GD-solvable and \(\cdot\) is commutative associative, there exists \(N\) such that
$
        A^{2^N}\subseteq A^{(N)}=0.
$
Hence \((A,\cdot)\) is nilpotent, and every \(L_z\) is nilpotent. Also, \(A_{\mathrm{Lie}}\) is solvable. By Lemma~\ref{lem:nilpotent-half-derivation},
$
        L_z(A)\subseteq \mathfrak n(A_{\mathrm{Lie}})
$
for all \(z\in A\). Therefore \(A\cdot A\subseteq \mathfrak n(A_{\mathrm{Lie}})\).
\end{proof}

\begin{lem}
\label{lem:TP-mixed-ideal-nilpotent}
Let \((A,\cdot,[\, ,\,])\) be a transposed Poisson algebra. If \(A_{\mathrm{Lie}}=(A,[\, ,\,])\) is nilpotent, then
$
        \mathcal I_{R\operatorname{ad}}(A)
$
is nilpotent.
\end{lem}

\begin{proof}
Since \(A_{\mathrm{Lie}}\) is nilpotent, there exists \(s \geq 1\) such that
\(\operatorname{ad}_{x_1}\operatorname{ad}_{x_2}\cdots\operatorname{ad}_{x_s}=0\) for all \(x_1,\ldots,x_s\in A\).
The transposed Poisson identity gives
$
        \operatorname{ad}_xL_y
        =
        2L_y\operatorname{ad}_x-\operatorname{ad}_{x\cdot y}.
$
Since the product is commutative, \(L_y=R_y\). Hence this relation moves every
\(\operatorname{ad}\)-factor to the right of every multiplication factor without decreasing the number of \(\operatorname{ad}\)-factors. Therefore, any monomial in \(\mathcal I_{R\operatorname{ad}}(A)^s\) is a linear combination of terms of the form \(L_{u_1}\cdots L_{u_p}\operatorname{ad}_{v_1}\cdots\operatorname{ad}_{v_t}\) with \(t\geq s\), all of which are zero. Thus
$
        \mathcal I_{R\operatorname{ad}}(A)^s=0.
$
\end{proof}

\begin{cor}
\label{cor:TP-solvable-iff-derived-nilpotent}
Let \((A,\cdot,[\, ,\,])\) be a transposed Poisson algebra. Then \(A\) is GD-solvable if and only if
$
        A^{(1)}=A\cdot A+[A,A]
$
is GD-nilpotent.
\end{cor}

\begin{proof}
Put
$
        B=A^{(1)}=A\cdot A+[A,A].
$
If \(B\) is GD-nilpotent, then \(B\) is GD-solvable. Since \(A/B\) has zero product and zero bracket, \(A\) is GD-solvable. Conversely, assume that \(A\) is GD-solvable. Since \(\cdot\) is commutative associative, there exists \(N\) such that
$
        A^{2^N}\subseteq A^{(N)}=0.
$
Thus \((A,\cdot)\) is nilpotent. Hence \((B,\cdot)\) is nilpotent as a Novikov algebra.

By Lemma~\ref{lem:TP-square-in-nilradical},
$
        A\cdot A\subseteq \mathfrak n(A_{\mathrm{Lie}}).
$
Moreover, \(A_{\mathrm{Lie}}\) is solvable, so
$
        [A,A]\subseteq \mathfrak n(A_{\mathrm{Lie}}).
$
Therefore
$
        B=A\cdot A+[A,A]\subseteq \mathfrak n(A_{\mathrm{Lie}}),
$
and \((B,[\, ,\,])\) is nilpotent as a Lie algebra. The subspace \(B\) is a transposed Poisson subalgebra of \(A\), since
$B \cdot B \subseteq A \cdot A \subseteq B$ and $[B, B] \subseteq [A, A] \subseteq B$. By Lemma~\ref{lem:TP-mixed-ideal-nilpotent} applied to \(B\),
$
        \mathcal I_{R\operatorname{ad}}(B)
$
is nilpotent. Now Proposition~\ref{prop:GD-derived-nilpotency-reduction}, applied to the GD algebra \(A\) with \(B=A^{(1)}\), implies that \(B\) is GD-nilpotent. The proof is complete.
\end{proof}

\section{Several Constructions of GD Algebras and Their Speciality}
\label{sec:constructions}

In this section, we present several methods for constructing GD algebras, and we determine whether the resulting GD algebras are special.

\subsection{Constructions of GD Algebras from Other Algebraic Structures}

\begin{prop}
\label{prop:Novikov_GD_construction}
\label{prop:Novikov_GD_special}
Let \((A,\circ)\) be a Novikov algebra and let
\(D\in\operatorname{Der}(A,\circ)\). Define
\[
        [x,y] := D(x) \circ y - D(y) \circ x
        \qquad \forall x,y \in A.
\]
Then \((A,\circ,[\cdot,\cdot])\) is a special GD algebra.
\end{prop}

\begin{proof}
By the embedding theorem for Novikov algebras, we may embed \(A\) into its universal differential commutative envelope \((P,\cdot,d)\), which is not required to be finite-dimensional, such that \(x\circ y=x\cdot d(y)\) for all \(x,y\in A\). By the universal property, \(D\) extends to a derivation \(\widetilde D\) of \(P\) satisfying \(\widetilde Dd=d\widetilde D\).

Define
$
        \{u,v\}:=\widetilde D(u)\cdot d(v)-\widetilde D(v)\cdot d(u)
$
for all \(u,v\in P\). Since \(d\) and \(\widetilde D\) are commuting derivations of the commutative associative algebra \(P\), this is a Poisson bracket, and \(d\) is a derivation of the bracket. Hence \((P,\cdot,\{\cdot,\cdot\},d)\) is a differential Poisson algebra.

For any \(x,y\in A\), we have
$
        \{x,y\}=D(x)\cdot d(y)-D(y)\cdot d(x)
        =D(x)\circ y-D(y)\circ x=[x,y],
$
and the Novikov product satisfies \(x\circ y=x\cdot d(y)\). Therefore the operations on \(A\) are induced from the differential Poisson algebra \((P,\cdot,\{\cdot,\cdot\},d)\). Thus \((A,\circ,[\cdot,\cdot])\) is a special GD algebra.
\end{proof}

\begin{prop}
\label{prop:square-zero-construction}
Let \((L,[\cdot,\cdot])\) be a Lie algebra and let \(N\in\operatorname{End}(L)\) satisfy
\[
        N^2=0,\quad [N(L),N(L)]=0,\quad N([N(L),L])=0.
\]
Define
$
        x\circ_N y:=N([x,y])-[N(x),y].
$
Then \((L,\circ_N,[\cdot,\cdot])\) is a GD algebra.
\end{prop}

\begin{proof}
We first verify the GD compatibility identity. For \(x,y,z\in L\), expanding the left-hand side of \eqref{eq:GD_compatibility} gives
\begin{equation*}
\begin{aligned}
&[x,y\circ_N z]-[z,y\circ_N x]+[y,x]\circ_N z
-[y,z]\circ_N x-y\circ_N[x,z]  \\
&\quad =
-N([[x,y],z])-N([[y,z],x])-N([y,[x,z]])  \\
&\qquad
-[x,[N(y),z]]+[z,[N(y),x]]+[N(y),[x,z]] .
\end{aligned}
\end{equation*}
The first three terms vanish by the Jacobi identity after applying \(N\), and the last three terms vanish by the Jacobi identity. Hence the GD compatibility identity holds.

It remains to prove that \(\circ_N\) is Novikov. From the assumptions we have, for all \(x,y\in L\),
$
        N(x)\circ_N y=0, x\circ_N N(y)=0.
$
Therefore
$
        (x\circ_N y)\circ_N z=-N([[N(x),y],z]).
$
It follows that
\begin{equation*}
\begin{aligned}
        (x\circ_N y)\circ_N z-(x\circ_N z)\circ_N y
        &=
        -N([[N(x),y],z])+N([[N(x),z],y])  \\
        &=
        -N([N(x),[y,z]])=0.
\end{aligned}
\end{equation*}
Thus the right-commutativity identity holds.

For the left-symmetric identity, let
$
        A(x,y,z):=(x\circ_N y)\circ_N z-x\circ_N(y\circ_N z).
$
Using \(x\circ_N N(u)=0\), a direct calculation gives
\begin{equation*}
\begin{aligned}
A(x,y,z)-A(y,x,z)
&=
N([N(y),[x,z]])-N([N(x),[y,z]])  \\
&\quad
-[N(x),[N(y),z]]+[N(y),[N(x),z]] .
\end{aligned}
\end{equation*}
The first two terms vanish by \(N([N(L),L])=0\), while the last two terms equal \(-[[N(x),N(y)],z]\), which is zero because \([N(L),N(L)]=0\). Hence \(A(x,y,z)=A(y,x,z)\). Therefore \(\circ_N\) is a Novikov product. Thus \((L,\circ_N,[\cdot,\cdot])\) is a GD algebra.
\end{proof}

The following example demonstrates that GD algebras obtained via the construction in Proposition~\ref{prop:square-zero-construction} are not necessarily special.

\begin{ex}
\label{ex:square-zero-solvable-nonspecial}
Let \(L=\operatorname{span}\{t,a,b,c\}\) be the Lie algebra with nonzero brackets
$
        [t,b]=-a+c$ and $ [t,c]=-c.
$
Then \(L\) is solvable. Define \(N\in\operatorname{End}(L)\) by
\[
        N(a)=b,\qquad N(c)=b,\qquad N(t)=N(b)=0.
\]
Then \(N^2=0\), \(N(L)=\Bbbk b\), \([N(L),N(L)]=0\), and
$
        [N(L),L]\subseteq \Bbbk(a-c)\subseteq \ker N.
$
Thus \(N([N(L),L])=0\), and Proposition~\ref{prop:square-zero-construction} gives a GD algebra \((L,\circ_N,[\cdot,\cdot])\).

The nonzero products are
\[
        a\circ_N t=-a+c,\qquad
        t\circ_N c=-b,\qquad
        c\circ_N t=-a+b+c.
\]

We show that this GD algebra is not special. Evaluate the special identity \eqref{eq:GD_identity_35} at
\[
        x_1=t,\qquad x_2=t,\qquad x_3=t,\qquad x_4=c.
\]
Its left-hand side is
\[
\begin{aligned}
        [t,t\circ_N c]\circ_N t
        +([t,c]\circ_N t)\circ_N t
        &=
        [t,-b]\circ_N t+((-c)\circ_N t)\circ_N t  \\
        &=
        (a-c)\circ_N t+(a-b-c)\circ_N t
        =
        -2b.
\end{aligned}
\]
Its right-hand side is
\[
\begin{aligned}
        [t,(t\circ_N c)\circ_N t]
        +[t\circ_N c,t]\circ_N t
        &=
        [t,(-b)\circ_N t]+[-b,t]\circ_N t  \\
        &=
        0+(-a+c)\circ_N t
        =
        b.
\end{aligned}
\]
Thus \eqref{eq:GD_identity_35} fails. Hence \((L,\circ_N,[\cdot,\cdot])\) is not special.
\end{ex}

\begin{prop}
\label{prop:gen-Poisson-GD-special}
Let \((A,\cdot,[\cdot,\cdot],D)\) be a generalized Poisson algebra. Define \(x\circ y=x\cdot D(y)\). Then \((A,\circ,[\cdot,\cdot])\) is a special GD algebra.
\end{prop}

\begin{proof}
We realize \((A,\circ,[\,,\,])\) as a GD subalgebra of a GD algebra arising from a differential Poisson algebra. Put
$
        P=A\otimes_{\Bbbk}\Bbbk[t,t^{-1}],
       $ and $
        x_m=x\otimes t^m .
$
This induced algebra \(P\) is not required to be finite-dimensional. Define the commutative associative product on \(P\) by
$
        x_m y_n=(x\cdot y)_{m+n}.
$
For homogeneous elements, define a bracket by
\begin{equation}
\label{eq:Poissonization-bracket}
\begin{aligned}
        \{x_m,y_n\}
        =
        \Big(
        [x,y]
        +
        (m-1)x\cdot D(y)
        -
        (n-1)D(x)\cdot y
        \Big)_{m+n-1}.
\end{aligned}
\end{equation}
This bracket is clearly skew-symmetric.

We first check the Leibniz rule. For \(x,y,z\in A\) and \(m,n,p\in \mathbb Z\), using the generalized Poisson compatibility \eqref{eq:Gen_Poisson_compatibility} and the fact that \(D\) is a derivation of \((A,\cdot)\), we have
\[
\begin{aligned}
\{x_m,(y\cdot z)_{n+p}\}
&=
\Big(
[x,y\cdot z]
+(m-1)x\cdot D(y\cdot z)  \\
&\qquad
-(n+p-1)D(x)\cdot y\cdot z
\Big)_{m+n+p-1}                                    \\
&=
\Big(
[x,y]\cdot z
+
y\cdot[x,z]                                      \\
&\qquad
+(m-1)x\cdot D(y)\cdot z
+(m-1)x\cdot y\cdot D(z)              \\
&\qquad
-(n+p-2)D(x)\cdot y\cdot z
\Big)_{m+n+p-1}                                  \\
&=
        \{x_m,y_n\}z_p+y_n\{x_m,z_p\}.
\end{aligned}
\]
Hence \(\{\,,\,\}\) is a derivation in each argument with respect to the associative product.

It remains to check the Jacobi identity. Set
$
        B_{m,n}(x,y)
        =
        [x,y]
        +
        (m-1)x\cdot D(y)
        -
        (n-1)D(x)\cdot y .
$
Then the Jacobi identity for \(\{\,,\,\}\) is equivalent to
\[
\begin{aligned}
& B_{m+n-1,p}\bigl(B_{m,n}(x,y),z\bigr)
 +B_{n+p-1,m}\bigl(B_{n,p}(y,z),x\bigr) 
 +B_{p+m-1,n}\bigl(B_{p,m}(z,x),y\bigr)
 =0 .
\end{aligned}
\]
Expanding this expression, the terms without \(D\) vanish by the Jacobi identity of \([\,,\,]\). The terms containing one occurrence of \(D\) vanish by
\eqref{eq:Gen_Poisson_compatibility} together with
$
        D([x,y])=[D(x),y]+[x,D(y)].
$
Finally, the remaining terms, containing products of \(D(x),D(y),D(z)\) or \(D^2(x),D^2(y)\), \(D^2(z)\) cancel by the commutativity and associativity of \((A,\cdot)\) and the derivation property of \(D\). Thus \((P,\cdot,\{\,,\,\})\) is a Poisson algebra.

Now define
$
        \partial:P\to P,
        \partial(x_m)=D(x)_{m-1}.
$
Since \(D\) is a derivation of \((A,\cdot)\), the map \(\partial\) is a derivation of the associative product of \(P\). Moreover, by \eqref{eq:Poissonization-bracket},
\[
\begin{aligned}
\partial\{x_m,y_n\}
&=
\Big(
[D(x),y]+[x,D(y)]
+(m-n)D(x)\cdot D(y)  \\
&\qquad
+(m-1)x\cdot D^2(y)
-(n-1)D^2(x)\cdot y
\Big)_{m+n-2}.
\end{aligned}
\]
The same expression is obtained by expanding
$
        \{\partial(x_m),y_n\}+\{x_m,\partial(y_n)\}.
$
Therefore
$
        \partial\{u,v\}
        =
        \{\partial u,v\}+\{u,\partial v\},
      \forall u,v\in P.
$
Hence \((P,\cdot,\{\,,\,\},\partial)\) is a differential Poisson algebra.

The associated special GD product on \(P\) is
$
        u\circ_P v=u\cdot \partial(v).
$
Define
$
        \iota:A\longrightarrow P,
        \iota(x)=x_1 .
$
Then \(\iota\) is injective, and for all \(x,y\in A\),
$
        \{\iota(x),\iota(y)\}
        =
        \{x_1,y_1\}
        =
        [x,y]_1
        =
        \iota([x,y]),
$
while
\[
        \iota(x)\circ_P \iota(y)
        =
        x_1\cdot \partial(y_1)
        =
        x_1\cdot D(y)_0
        =
        (x\cdot D(y))_1
        =
        \iota(x\circ y).
\]
Thus \(\iota(A)\) is a GD subalgebra of the special GD algebra associated to the differential Poisson algebra \((P,\cdot,\{\,,\,\},\partial)\). Consequently \((A,\circ,[\,,\,])\) is special.
\end{proof}

\subsection{Constructions of Larger GD Algebras from smaller Ones}

\begin{prop}\label{prop:n_GD_construction}
Let \((A,\circ,[\cdot,\cdot])\) be an \(n\)-dimensional GD algebra and \(d:A\to A\) a linear map which is a derivation of both \((A,[\cdot,\cdot])\) and \((A,\circ)\), and satisfies \(d(x)\circ y=d(y)\circ x\) for any \(x,y\in A\). On
\(\bar A:=A\oplus\langle h\rangle\), define
\[
        [A,h]=0, \quad x\circ h=0, \quad h\circ x=d(x), \quad h\circ h=0 \quad \text{for all } x\in A.
\]
Then \((\bar A,\circ,[\cdot,\cdot])\) is an \((n+1)\)-dimensional GD algebra.
\end{prop}

\begin{proof}
For \(X=a+\alpha h\) and \(Y=b+\beta h\), where \(a,b\in A\), we have
$
        X\circ Y=a\circ b+\alpha d(b)$ and $
        [X,Y]=[a,b].
$
Thus the extended bracket is a Lie bracket, since \(h\) is central.

The Novikov identities on \(\bar A\) reduce to those on \(A\), together with the derivation property of \(d\) and the condition \(d(x)\circ y=d(y)\circ x\) for any \(x,y\in A\). More precisely, right-commutativity follows from
$
        (a\circ b+\alpha d(b))\circ c
        =
        (a\circ c+\alpha d(c))\circ b,
$
while the left-symmetric identity follows by expanding the associator and using that \(d\) is a derivation of \((A,\circ)\).

Finally, the GD compatibility for \(X=a+\alpha h\), \(Y=b+\beta h\) and \(Z=c+\gamma h\) equals the GD compatibility expression for \(a,b,c\in A\), plus
$
        \beta\bigl([a,d(c)]-[c,d(a)]-d([a,c])\bigr),
$
which vanishes because \(A\) is a GD algebra and \(d\) is a derivation of the Lie bracket. Hence \((\bar A,\circ,[\cdot,\cdot])\) is a GD algebra.
\end{proof}

\begin{prop}
\label{prop:inverse_n_GD_construction}
Let \((\bar A,\circ,[\cdot,\cdot])\) be a GD algebra, and set \(\bar A^2:=\bar A\circ\bar A+[\bar A,\bar A]\). Suppose that there exists \(h\in \bar A\) such that \([\bar A,h]=0\), \(\bar A\circ h=0\) and \(h\notin \bar A^2\). Then \(\bar A\) is obtained by the construction in Proposition~\ref{prop:n_GD_construction}.
\end{prop}

\begin{proof}
Choose a subspace \(A\subseteq \bar A\) such that \(\bar A=A\oplus\langle h\rangle\) and \(\bar A^2\subseteq A\). Since \(A\) contains \(\bar A^2\), it is a GD ideal of \(\bar A\). Define \(d:A\to A\) by \(d(x)=h\circ x\). Then the restrictions of the operations to \(A\) make \(A\) a GD algebra, and \(\bar A\) has the form
\[
        [A,h]=0,\quad
        x\circ h=0,\quad
        h\circ x=d(x),\quad
        h\circ h=0,\quad
       \text{for all } x\in A. 
\]

It remains to check that \(d\) satisfies the hypotheses of Proposition~\ref{prop:n_GD_construction}. Applying the GD compatibility identity
with \(y=h\) gives
$
        [x,d(z)]-[z,d(x)]-d([x,z])=0,
$
so \(d\) is a derivation of the Lie algebra \((A,[\cdot,\cdot])\). The Novikov
right-commutativity identity with first argument \(h\) gives \(d(x)\circ y=d(y)\circ x\). Finally, the left-symmetric identity with first argument \(h\) gives
$
        d(x\circ y)=d(x)\circ y+x\circ d(y),
$
so \(d\) is a derivation of the Novikov algebra \((A,\circ)\). Therefore \(\bar A\) is precisely obtained from \(A\) and \(d\) by the construction in Proposition~\ref{prop:n_GD_construction}.
\end{proof}

We now present a example demonstrating that above extension of a GD algebra can destroy the special property. Specifically, we will construct a larger non-special GD algebra starting from a special one.

\begin{ex}\label{ex:non-special-extension}
Let \(\mathcal P=\Bbbk[x_1,x_2,x_3]\) be the infinite-dimensional polynomial algebra. Define a Poisson bracket on \(\mathcal P\) by
$
        \{x_1,x_2\}=x_2,
        \{x_3,x_1\}=1$ and $
        \{x_2,x_3\}=0.
$
And extend it by the Leibniz rule. The Jacobi identity is checked on the generators. Let \(\partial=\partial/\partial x_1\). Then \(\partial\) is a Poisson derivation, and hence
$
        f\circ g:=f\,\partial(g)$ and $ [f,g]:=\{f,g\}
$
defines a special GD algebra on \(\mathcal P\).

Set \(D:=x_3\partial\). A direct check on the generators shows that \(D\) is a
derivation of the Lie algebra \((\mathcal P,[\cdot,\cdot])\). Moreover,
$
        D(f\circ g)=D(f)\circ g+f\circ D(g)$ and $
        D(f)\circ g=D(g)\circ f
$
for all \(f,g\in\mathcal P\). Thus \(D\) satisfies the hypotheses of Proposition~\ref{prop:n_GD_construction} in its dimension-free form. Define \(\overline{\mathcal P}:=\mathcal P\oplus\langle h\rangle\) by
\[
        [\mathcal P,h]=0,\quad
        f\circ h=0,\quad
        h\circ f=D(f),\quad
        h\circ h=0,\quad
        \text{for all } f\in\mathcal P.
\]
Then \((\overline{\mathcal P},\circ,[\cdot,\cdot])\) is a GD algebra.

We claim that this GD algebra is not special. Evaluate \eqref{eq:GD_identity_35} at
\[
        x_1=x_1, \quad x_2=x_1,\quad  x_3=h,\quad
        x_4=\frac{1}{2}x_1^2 .
\]
Since \(h\circ x_1=x_3\) and \([h,x_1]=0\), the left-hand side equals
\[
        [x_1,h\circ x_1]\circ \frac{1}{2}x_1^2
        =
        [x_1,x_3]\circ \frac{1}{2}x_1^2
        =
        -1\circ \frac{1}{2}x_1^2
        =
        -x_1 .
\]
On the other hand, \(h\circ \frac{1}{2}x_1^2=x_3x_1\), and the right-hand side
equals
\[
        [x_1,x_3x_1]+[x_3x_1,x_1]\circ x_1
        =
        -x_1+x_1
        =
        0 .
\]
Thus the special identity \eqref{eq:GD_identity_35} fails. Therefore \(\overline{\mathcal P}\) is not a special GD algebra.
\end{ex}

\begin{prop}
\label{prop:GD_holomorph}
Let \((A,\circ,[\cdot,\cdot])\) be a GD algebra. Define
\[
        \operatorname{Der}_s(A)
        :=
        \{D\in\operatorname{End}(A)\mid D \text{ is a derivation of both }
        (A,\circ) \text{ and } (A,[\cdot,\cdot]),\]
\[        
        \ D(x)\circ y=D(y)\circ x \text{ for all } x,y\in A\}.
\]
On the vector space \(A\rtimes \operatorname{Der}_s(A)\), define
$
        [x+f,y+g]:=[x,y]+f(y)-g(x)+fg-gf
$
and
$
        (x+f)\circ(y+g):=x\circ y
$
for all \(x,y\in A\) and \(f,g\in\operatorname{Der}_s(A)\). Then \((A\rtimes \operatorname{Der}_s(A),\circ,[\cdot,\cdot])\) is a GD algebra.
\end{prop}

\begin{proof}
First, \(\operatorname{Der}_s(A)\) is closed under commutators. Indeed, the commutator of two derivations is again a derivation of both operations, and for
\(f,g\in\operatorname{Der}_s(A)\) we have
\[
\begin{aligned}
        (fg-gf)(x)\circ y
        &=f(g(x))\circ y-g(f(x))\circ y  \\
        &=f(g(x)\circ y)-g(x)\circ f(y)-g(y)\circ f(x)  \\
        &=f(g(y)\circ x)-g(f(y))\circ x-g(y)\circ f(x)  \\
        &=(fg-gf)(y)\circ x .
\end{aligned}
\]
Thus the displayed bracket is the usual semidirect product Lie bracket.

The product on \(A\rtimes\operatorname{Der}_s(A)\) is Novikov, since it depends
only on the \(A\)-components and \((A,\circ)\) is Novikov. It remains to check
the GD compatibility. For \(x,y,z\in A\) and \(f,g,w\in\operatorname{Der}_s(A)\), the GD expression for \(x+f, y+g\) and \(z+w\) is the GD expression for \(x,y,z\in A\), plus the terms
\[
        f(y\circ z)-f(y)\circ z-y\circ f(z),\quad
         -w(y\circ x)+w(y)\circ x+y\circ w(x),\quad
           g(x)\circ z-g(z)\circ x.
\]
The first two vanish because \(f\) and \(w\) are derivations of \((A,\circ)\), while the last one vanishes by the defining condition of \(\operatorname{Der}_s(A)\). Hence the GD compatibility identity holds. Therefore \(A\rtimes\operatorname{Der}_s(A)\) is a GD algebra.
\end{proof}

In the following, we provide concrete examples illustrating how to construct non-special GD algebras from both finite-dimensional and infinite-dimensional cases, using the construction in Proposition~\ref{prop:GD_holomorph}.

\begin{ex}\label{ex:semidirect-nonspecial-fd}
Let \(A=\operatorname{span}\{x,y\}\). Define
$
        [x,y]=y,
        x\circ x=x,
        y\circ x=y.
$
Then \((A,\circ,[\cdot,\cdot])\) is a 2-dimensional GD algebra. By \cite{Kolesnikov-Sartayev}, every 2-dimensional GD algebra is special, hence \(A\) is special.

Define \(\delta:A\to A\) by \(\delta(x)=y\) and \(\delta(y)=0\). A direct check shows that \(\delta\) is a derivation of both \((A,[\cdot,\cdot])\) and \((A,\circ)\), and that \(\delta(u)\circ v=\delta(v)\circ u\) for all \(u,v\in A\). Hence \(\delta\in\operatorname{Der}_s(A)\), and Proposition~\ref{prop:GD_holomorph} gives a GD algebra structure on \(A\rtimes\operatorname{Der}_s(A)\).

We show that this GD algebra is not special. Evaluate \eqref{eq:GD_identity_35} at
\[
        x_1=x,\quad x_2=x,\quad x_3=\delta,\quad x_4=x,
\]
where \(x\) and \(\delta\) are identified with \(x+0\) and \(0+\delta\), respectively. Since \(\delta\circ x=0\), the first term on the left-hand side vanishes. Moreover, \([\delta,x]=\delta(x)=y\), and hence
$
        ([\delta,x]\circ x)\circ x=(y\circ x)\circ x=y.
$
Thus the left-hand side of \eqref{eq:GD_identity_35} equals \(y\). On the other hand, the right-hand side is
$
        [x,(\delta\circ x)\circ x]+[\delta\circ x,x]\circ x=0.
$
Therefore \eqref{eq:GD_identity_35} fails. Since every special GD algebra satisfies \eqref{eq:GD_identity_35}, the GD algebra \(A\rtimes\operatorname{Der}_s(A)\) is not special.
\end{ex}

\begin{ex}\label{ex:semidirect-nonspecial-infdim}
Let \(P=\Bbbk[x]\) and put \(\partial=x\frac{d}{dx}\). Define $[f,g]=0$ and $f\circ g=f\,\partial(g)$ for all $f,g \in P$. Then \((P,\circ,[\cdot,\cdot])\) is an infinite-dimensional GD algebra. Indeed, it is induced by the differential commutative algebra \((P,\cdot,\partial)\) with zero Poisson bracket.

Now set \(D:=\partial\). Since the Lie bracket is zero, \(D\) is trivially a derivation of the Lie algebra. Moreover,
$
        D(f\circ g)=D(f)\circ g+f\circ D(g)$ and $
        D(f)\circ g=D(g)\circ f
$
for all \(f,g\in P\). Hence \(D\in\operatorname{Der}_s(P)\), and Proposition~\ref{prop:GD_holomorph} gives a GD algebra structure on \(P\rtimes\operatorname{Der}_s(P)\).

We show that this GD algebra is not special. Evaluate \eqref{eq:GD_identity_35} at
\[
        x_1=x,\quad x_2=x,\quad x_3=D,\quad x_4=x.
\]
Since \(D\circ x=0\), the first term on the left-hand side is zero. Moreover,
$
        [D,x]=D(x)=x,
$
and therefore
\[
        ([D,x]\circ x)\circ x=(x\circ x)\circ x=x^3.
\]
Thus the left-hand side of \eqref{eq:GD_identity_35} equals \(x^3\). On the other hand, since \(D\circ x=0\), the right-hand side is
\[
        [x,(D\circ x)\circ x]+[D\circ x,x]\circ x=0.
\]
Hence \eqref{eq:GD_identity_35} fails. Since every special GD algebra satisfies \eqref{eq:GD_identity_35}, the GD algebra \(P\rtimes\operatorname{Der}_s(P)\) is not special.
\end{ex}

\section{GD Algebra Structures on Simple Lie Algebras}
\label{sec:simple-lie}

In this section, we study GD algebra structures on simple Lie algebras. It is well known that Poisson and transposed Poisson structures on simple Lie algebras are trivial\cite{FernandezOuaridi2024,BenayadiBoucetta2014}, however, as we shall see, GD algebra structures do not share this property.

\begin{prop}
\label{prop:cocycle}
Let \((L,[\cdot,\cdot])\) be a Lie algebra and let \(\circ\) be a bilinear product on \(L\). Let \(L\) act on \(\operatorname{End}(L)\) by \(x\cdot S=[\operatorname{ad}_x,S]\), where \(\operatorname{ad}_x(y)=[x,y]\). Then the GD compatibility identity \eqref{eq:GD_compatibility} is equivalent to the condition that \(R:L\to\operatorname{End}(L)\), \(a\mapsto R_a\), is a Chevalley--Eilenberg
\(1\)-cocycle, namely
\begin{equation}
\label{eq:cocycle}
        R_{[x,z]}=[\operatorname{ad}_x,R_z]-[\operatorname{ad}_z,R_x]
        \qquad \forall x,z \in L.
\end{equation}
Moreover, \(\circ\) is a Novikov product if and only if
\begin{align}
        [R_x,R_y]&=0,                         \label{eq:Rcomm}\\
        [L_x,L_y]&=L_{x\circ y-y\circ x}       \label{eq:Lleft}
\end{align}
for all \(x,y\in L\).
\end{prop}

\begin{proof}
For \(x,y,z\in L\), we have
\[
\begin{aligned}
&\bigl([\operatorname{ad}_x,R_z]-[\operatorname{ad}_z,R_x]-R_{[x,z]}\bigr)y  \\
&\quad =
[x,y\circ z]-[z,y\circ x]+[y,x]\circ z-[y,z]\circ x-y\circ[x,z].
\end{aligned}
\]
Thus \eqref{eq:cocycle} is precisely the GD compatibility identity.

The last assertion follows from the standard operator form of the two Novikov identities. Indeed, \([R_x,R_y]=0\) is equivalent to right-commutativity, while
\([L_x,L_y]=L_{x\circ y-y\circ x}\) is equivalent to the left-symmetric identity. Hence \(\circ\) is Novikov if and only if \eqref{eq:Rcomm} and \eqref{eq:Lleft} hold.
\end{proof}

\begin{thm}
\label{thm:T-reduction}
Let \(\mathfrak g\) be a finite-dimensional complex simple Lie algebra. A bilinear product \(\circ\) makes \((\mathfrak g,\circ,[\cdot,\cdot])\) a GD algebra if and only if there exists \(T\in\operatorname{End}(\mathfrak g)\) such that
\begin{equation}
\label{eq:T-product}
        x\circ y=T([x,y])-[T(x),y]
        \qquad \forall x,y \in\mathfrak g,
\end{equation}
and the product defined by \eqref{eq:T-product} is Novikov.

Moreover, two endomorphisms \(T,T'\) define the same product if and only if
$
        T'-T\in \mathbb C\,\operatorname{id}_{\mathfrak g}.
$
\end{thm}

\begin{proof}
By Proposition~\ref{prop:cocycle}, the GD compatibility identity is equivalent to the assertion that \(R:\mathfrak g\to\operatorname{End}(\mathfrak g)\), \(x\mapsto R_x\), is a \(1\)-cocycle for the \(\mathfrak g\)-module \(\operatorname{End}(\mathfrak g)\) with action \(x\cdot S=[\operatorname{ad}_x,S]\). Since \(\mathfrak g\) is semisimple and \(\operatorname{End}(\mathfrak g)\) is finite-dimensional, Whitehead's first lemma gives
$
        H^1(\mathfrak g,\operatorname{End}(\mathfrak g))=0.
$
Hence every such cocycle is a coboundary. Thus there exists \(T\in\operatorname{End}(\mathfrak g)\) such that \(R_y=[\operatorname{ad}_y,T]\) for all \(y\in\mathfrak g\). Therefore, for \(x,y\in\mathfrak g\),
$
        x\circ y
        =
        R_y(x)
        =
        [y,T(x)]-T([y,x])
        =
        T([x,y])-[T(x),y].
$
Conversely, any product defined by \eqref{eq:T-product} satisfies \(R_y=[\operatorname{ad}_y,T]\). Hence its right multiplication map is a coboundary, and so the GD compatibility identity holds. Thus the remaining condition is precisely that the product \eqref{eq:T-product} be Novikov.

Finally, \(T\) and \(T'\) define the same product if and only if
\([\operatorname{ad}_y,T'-T]=0\) for all \(y\in\mathfrak g\). Hence \(T'-T\) is an endomorphism of the adjoint \(\mathfrak g\)-module. Since \(\mathfrak g\) is simple, the adjoint module is irreducible. By Schur's lemma, \(T'-T\in\mathbb C\,\operatorname{id}_{\mathfrak g}\). Conversely, if \(T'-T\in\mathbb C\,\operatorname{id}_{\mathfrak g}\), then \eqref{eq:T-product} gives the same product. Therefore GD products on the fixed Lie algebra \(\mathfrak g\) are parametrized by the classes
$
        [T]\in
        \operatorname{End}(\mathfrak g)/
        \mathbb C\,\operatorname{id}_{\mathfrak g}
$
for which the product \eqref{eq:T-product} is Novikov.
\end{proof}

\begin{prop}
\label{prop:right-nilpotent}
Let \((L,\circ,[\cdot,\cdot])\) be a finite-dimensional complex GD algebra. If the Lie algebra \((L,[\cdot,\cdot])\) is perfect, that is, \([L,L]=L\), then every right multiplication operator \(R_a\) is nilpotent.
\end{prop}

\begin{proof}
By Proposition~\ref{prop:cocycle}, the GD compatibility identity is equivalent
to
$
        R_{[x,y]}=[\operatorname{ad}_x,R_y]-[\operatorname{ad}_y,R_x].
$
Since \((L,\circ)\) is Novikov, the right multiplications commute pairwise. Let \(P\in\operatorname{End}(L)\) commute with all \(R_a\). Using the cyclicity of the trace, we have
$
        \operatorname{tr}([\operatorname{ad}_x,R_y]P)
        =
        \operatorname{tr}(\operatorname{ad}_xR_yP-\operatorname{ad}_xPR_y)=0,
$
and similarly \(\operatorname{tr}([\operatorname{ad}_y,R_x]P)=0\). Hence
$
        \operatorname{tr}(R_{[x,y]}P)=0
$
for all \(x,y\in L\). Since \(L=[L,L]\), it follows that
\(\operatorname{tr}(R_aP)=0\) for all \(a\in L\) and each such \(P\).

Now fix \(a\in L\). Taking \(P=R_a^{m-1}\), which commutes with all right
multiplications, gives
$
        \operatorname{tr}(R_a^m)=0
$
for \(m\geq 1\). Consequently, all power sums of the eigenvalues of \(R_a\) vanish. By Newton's identities, all eigenvalues of \(R_a\) are zero, and hence \(R_a\) is nilpotent.
\end{proof}

\begin{cor}
\label{cor:T-nilpotent-condition}
Let \(\mathfrak g\) be a finite-dimensional complex simple Lie algebra, and let \(T\in \operatorname{End}(\mathfrak g)\) define a GD product by
$
        x\circ_T y=T([x,y])-[T(x),y].
$
Then
$
        R_y^T=[\operatorname{ad}_y,T]
$
is nilpotent for every \(y\in\mathfrak g\), where \(\operatorname{ad}_y(z)=[y,z]\). Equivalently,
$
        \operatorname{tr}\bigl([\operatorname{ad}_y,T]^m\bigr)=0,
        \text{where }1\leq m\leq \dim\mathfrak g,\ y\in\mathfrak g.
$
\end{cor}

\begin{proof}
Since \(\mathfrak g\) is simple, it is perfect. Hence Proposition~\ref{prop:right-nilpotent} implies that every right multiplication operator \(R_y^T\) is nilpotent.
Moreover, for \(z\in\mathfrak g\), the definition of \(\circ_T\) gives
$
        R_y^T(z)
        =
        z\circ_T y
        =
        T([z,y])-[T(z),y]
        =
        [y,T(z)]-T([y,z])
        =
        [\operatorname{ad}_y,T](z).
$
Thus \(R_y^T=[\operatorname{ad}_y,T]\). The trace formulation is equivalent to nilpotency over \(\mathbb C\) by Newton's identities, since the traces of the first
\(\dim\mathfrak g\) powers determine the coefficients of the characteristic polynomial.
\end{proof}

We now classify the GD products whose associated class in \(\operatorname{End}(\mathfrak g)/\mathbb C\,\operatorname{id}_{\mathfrak g}\) has a rank-one representative. If a product is represented by \(T\in\operatorname{End}(\mathfrak g)\) as in Theorem~\ref{thm:T-reduction}, then its associated class is
$
        [T]=T+\mathbb C\,\operatorname{id}_{\mathfrak g}.
$
We say that this class has a rank-one representative if \(T+\lambda\operatorname{id}_{\mathfrak g}\) has rank one for some \(\lambda\in\mathbb C\). This is well-defined because adding a scalar multiple of \(\operatorname{id}_{\mathfrak g}\) to \(T\) does not change the product.

\begin{thm}
\label{thm:rank-one}
Let \(\mathfrak g\) be a finite-dimensional complex simple Lie algebra. Let \(u\in\mathfrak g\) and \(\varphi\in\mathfrak g^*\) be nonzero, and set \(T=\varphi\otimes u\), that is, \(T(x)=\varphi(x)u\). Then the product
\begin{equation}
\label{eq:rank-one-product}
        x\circ_T y
        =
        \varphi([x,y])u-\varphi(x)[u,y]
\end{equation}
defines a GD algebra structure on \(\mathfrak g\) if and only if
\begin{equation}
\label{eq:rank-one-condition}
        \varphi(u)=0,\quad \varphi([u,\mathfrak g])=0.
\end{equation}
\end{thm}

\begin{proof}
Assume first that \eqref{eq:rank-one-condition} holds. Then \(T^2=0\), \(T(\mathfrak g)=\mathbb C u\) is abelian, and \(T([T(\mathfrak g),\mathfrak g])=0\). Hence \(T\) satisfies the hypotheses of Proposition~\ref{prop:square-zero-construction}, and therefore \eqref{eq:rank-one-product} defines a GD algebra.

Conversely, suppose that \eqref{eq:rank-one-product} defines a GD algebra. For
\(y\in\mathfrak g\), set \(v_y=[u,y]\), \(\beta_y(x)=\varphi([x,y])\)  and
\(a=\varphi(u)\). Then
$
        R_y=u\otimes\beta_y-v_y\otimes\varphi .
$
Since the product is Novikov, \([R_y,R_z]=0\). Evaluating this equality on \(u\) gives
$
        0=[R_y,R_z]u
        =
        a\,\varphi([u,[y,z]])u .
$
If \(a\neq0\), then \(\varphi([u,\mathfrak g])=0\), because \(\mathfrak g=[\mathfrak g,\mathfrak g]\). Under this condition, \([R_y,R_z]=0\) reduces to
$
        v_z\otimes\beta_y=v_y\otimes\beta_z$ for all $ y,z\in \mathfrak g.
$
Since \(u\neq0\) and \(\mathfrak g\) is simple, \(\operatorname{ad}_u\neq0\).
Moreover, \(\operatorname{ad}_u\) is skew-adjoint with respect to the Killing form, so \(\dim [u,\mathfrak g]\geq2\). Choose \(y,z\) such that \(v_y\) and \(v_z\) are linearly independent. The tensor identity forces \(\beta_y=\beta_z=0\), and then, keeping one of these elements fixed, gives \(\beta_w=0\) for all \(w\in\mathfrak g\). Thus \(\varphi([\mathfrak g,\mathfrak g])=0\), hence \(\varphi=0\), a contradiction. Therefore \(\varphi(u)=0\).

It remains to prove \(\varphi([u,\mathfrak g])=0\). Since \(\mathfrak g\) is perfect, Proposition~\ref{prop:right-nilpotent} implies that every \(R_y\) is nilpotent. Hence \(\operatorname{tr}(R_y^2)=0\). Using \(\varphi(u)=0\) and \(\operatorname{tr}(w\otimes\alpha)=\alpha(w)\), we obtain
$
        \operatorname{tr}(R_y^2)
        =
        2\varphi([u,y])^2 .
$
Therefore \(\varphi([u,y])=0\) for all \(y\in\mathfrak g\). This proves \eqref{eq:rank-one-condition}.
\end{proof}

Let \(B\) be a nonzero invariant symmetric bilinear form on \(\mathfrak g\). Since \(\mathfrak g\) is simple, \(B\) is nondegenerate. We identify \(\mathfrak g\) with \(\mathfrak g^*\) by \(v\mapsto B(\,\cdot\,,v)\).

\begin{cor}
\label{cor:rank-one-B-form}
Let \(\mathfrak g\) be a finite-dimensional complex simple Lie algebra. The nonzero GD products whose associated class in \(\operatorname{End}(\mathfrak g)/\mathbb C\operatorname{id}_{\mathfrak g}\) has a rank-one representative are precisely the products
\begin{equation}
\label{eq:rank-one-B-product}
        x\circ_{u,v} y
        =
        B([x,y],v)u-B(x,v)[u,y],
\end{equation}
where \(u,v\in\mathfrak g\) are nonzero elements satisfying
\begin{equation}
\label{eq:commuting-isotropic-pair}
        [u,v]=0,\quad B(u,v)=0.
\end{equation}
\end{cor}

\begin{proof}
Let \(T=\varphi\otimes u\) be a nonzero rank-one representative. Since \(B\) is nondegenerate, \(\varphi(x)=B(x,v)\) for a unique nonzero \(v\in\mathfrak g\). By Theorem~\ref{thm:rank-one}, the corresponding product is GD if and only if
$
        \varphi(u)=0, \varphi([u,\mathfrak g])=0.
$
These conditions become \(B(u,v)=0\) and \(B([u,x],v)=0\) for all \(x\in\mathfrak g\). By invariance of \(B\), the second condition is equivalent to \(B([v,u],x)=0\) for all \(x\), hence to \([u,v]=0\). Thus the conditions are exactly \eqref{eq:commuting-isotropic-pair}, and \eqref{eq:rank-one-B-product} follows from
\(x\circ y=T([x,y])-[T(x),y]\).

Conversely, if \(u,v\neq0\) satisfy \eqref{eq:commuting-isotropic-pair}, then \(T_{u,v}(x)=B(x,v)u\) satisfies the conditions of Theorem~\ref{thm:rank-one}. Hence \eqref{eq:rank-one-B-product} defines a GD product. It is nonzero because \(T_{u,v}\) is rank one and therefore is not a scalar multiple of \(\operatorname{id}_{\mathfrak g}\).
\end{proof}

\begin{rem}
For two pairs \((u,v)\) and \((u',v')\) satisfying \eqref{eq:commuting-isotropic-pair}, we have
$
        \circ_{u,v}=\circ_{u',v'}
$
if and only if there exists \(c\in\mathbb C^*\) such that \(u'=cu\) and \(v'=c^{-1}v\). Indeed, equality of the products implies, by Theorem~\ref{thm:T-reduction}, that \(T_{u,v}-T_{u',v'}\in\mathbb C\operatorname{id}_{\mathfrak g}\). Since both operators have rank one and \(\dim\mathfrak g\geq3\), this scalar must be zero. Thus \(T_{u,v}=T_{u',v'}\), which is equivalent to the stated rescaling.
Moreover, every automorphism of \(\mathfrak g\) preserves \(B\) and sends \(\circ_{u,v}\) to \(\circ_{\sigma(u),\sigma(v)}\). Hence the isomorphism classes, under Lie algebra automorphisms of \(\mathfrak g\), of nonzero rank-one-representable GD products are parametrized by the \(\operatorname{Aut}(\mathfrak g)\)-orbits on
$
        \{(u,v)\in\mathfrak g\times\mathfrak g
        \mid u\neq0,\ v\neq0,\ [u,v]=0,\ B(u,v)=0\}/\mathbb C^*,
$
where \(c\cdot(u,v)=(cu,c^{-1}v)\).
\end{rem}

We now classify all GD products on the fixed Lie algebra \(\mathfrak{sl}_2(\mathbb C)\). We first show that every nonzero GD product on \(\mathfrak{sl}_2(\mathbb C)\) belongs to the rank-one case considered above.

\begin{lem}
\label{lem:sl2-rank-one-reduction}
Let \(\mathfrak g=\mathfrak{sl}_2(\mathbb C)\). Let \(\circ\) be a nonzero GD product on the fixed Lie algebra \(\mathfrak g\), and let
$
        [T]\in
        \operatorname{End}(\mathfrak g)/
        \mathbb C\operatorname{id}_{\mathfrak g}
$
be the class associated with \(\circ\) by Theorem~\ref{thm:T-reduction}. Then \([T]\) contains a nonzero rank-one endomorphism.
\end{lem}

\begin{proof}
Choose the representative \(T\) with \(\operatorname{tr}T=0\). By Theorem~\ref{thm:T-reduction}, the right multiplications are \(R_y=[\operatorname{ad}_y,T]\). Since \((\mathfrak g,\circ)\) is Novikov, the operators \(R_y\) commute pairwise, and by Proposition~\ref{prop:right-nilpotent}, they are nilpotent. Thus
$
        \{[\operatorname{ad}_y,T]\mid y\in\mathfrak g\}
$
is a commutative subspace of \(\operatorname{End}(\mathfrak g)\) consisting of nilpotent endomorphisms. Since \(\dim\mathfrak g=3\), such a subspace has dimension at most \(2\). Hence
$
        \dim\{[\operatorname{ad}_y,T]\mid y\in\mathfrak g\}\leq 2.
$

If this dimension is \(0\), then \(T\) commutes with the adjoint representation, so \(T\) is scalar. Since \(\operatorname{tr}T=0\), this gives \(T=0\), contrary to the nonzero product assumption. If the dimension is \(1\), then the kernel of \(y\mapsto [\operatorname{ad}_y,T]\) is a two-dimensional subalgebra of \(\mathfrak{sl}_2(\mathbb C)\). Hence it is conjugate to a Borel subalgebra. Thus \(T\) commutes with both \(\operatorname{ad}_h\) and \(\operatorname{ad}_e\), for a standard basis \(e,h,f\), which forces \(T\) to be scalar, again a contradiction.

Therefore the dimension is \(2\), and the kernel is \(\mathbb C u\) for some nonzero \(u\in\mathfrak g\). If \(u\) is semisimple, we may take \(u=h\). Then \(T\) is diagonal in the basis \(e,h,f\), set
$
        T(e)=\alpha e, T(h)=\beta h$ and $T(f)=\gamma f.
$
A short computation gives
$
        [R_e,R_f]e=-2(\beta-\alpha)^2e$ and $
        [R_e,R_f]f=2(\gamma-\beta)^2f.
$
Since \([R_e,R_f]=0\), we obtain \(\alpha=\beta=\gamma\), so \(T\) is scalar, a contradiction. Hence \(u\) is nilpotent.

After an automorphism, take \(u=e\). Since \(T\) commutes with \(\operatorname{ad}_e\), and \(\operatorname{ad}_e\) is a single nilpotent Jordan block on the adjoint module, modulo scalar endomorphisms one has
$
        T=a\operatorname{ad}_e+b\operatorname{ad}_e^2.
$
Then
$
        \operatorname{tr}([\operatorname{ad}_f,T]^2)=8a^2.
$
But \([\operatorname{ad}_f,T]=R_f\) is nilpotent, hence \(a=0\). Therefore, modulo scalar endomorphisms,
$
        T=b\operatorname{ad}_e^2.
$
Since the product is nonzero, \(b\neq0\). As \(\operatorname{ad}_e^2\) has rank one, the class \([T]\) contains a nonzero rank-one representative.
\end{proof}

We can now give the complete classification.

\begin{thm}
\label{thm:sl2-classification}
Let \(\mathfrak g=\mathfrak{sl}_2(\mathbb C)\) with standard basis
$
        [h,e]=2e, [h,f]=-2f,[e,f]=h.
$
Normalize the invariant symmetric bilinear form \(B\) by
\(B(e,f)=1\) and \(B(h,h)=2\). Set
$
        \mathcal N:=\{u\in\mathfrak g\mid B(u,u)=0\}.
$
For \(u\in\mathcal N\), define
$
        T_u(x):=B(x,u)u,
        x\circ_u y:=T_u([x,y])-[T_u(x),y].
$
Then the assignment \(u\mapsto \circ_u\) induces a bijection from \(\mathcal N/\{\pm1\}\) onto the set of all GD products on the fixed Lie algebra \(\mathfrak{sl}_2(\mathbb C)\).
\end{thm}

\begin{proof}
Let \(u\in\mathcal N\). If \(u=0\), then \(\circ_u=0\). If \(u\neq0\), then \(T_u^2=0\), \(T_u(\mathfrak g)=\mathbb C u\) is abelian, and, by invariance of
\(B\),
$
        B([u,x],u)=0,
        \text{for } x\in\mathfrak g.
$
Hence \(T_u([u,\mathfrak g])=0\). Thus \(T_u\) satisfies the hypotheses of Proposition~\ref{prop:square-zero-construction}, and \(\circ_u\) is a GD product.

Conversely, let \(\circ\) be a GD product on \(\mathfrak g\). If \(\circ=0\), then \(\circ=\circ_0\). Assume \(\circ\neq0\). By Lemma~\ref{lem:sl2-rank-one-reduction}, the associated class in
$
        \operatorname{End}(\mathfrak g)/
        \mathbb C\operatorname{id}_{\mathfrak g}
$
contains a nonzero rank-one representative. Hence Corollary~\ref{cor:rank-one-B-form} applies, and the product is represented by \(T_{u,v}(x)=B(x,v)u\), where \(u,v\neq0\), \([u,v]=0\), and \(B(u,v)=0\). Since the centralizer of every nonzero element of \(\mathfrak{sl}_2(\mathbb C)\) is one-dimensional, \(v=\lambda u\) for some \(\lambda\in\mathbb C^*\). Hence \(B(u,v)=0\) gives \(B(u,u)=0\). Choosing
\(c\in\mathbb C^*\) with \(c^2=\lambda\), we have
$
        T_{u,v}=\lambda T_u=T_{cu},
$
so the product is of the form \(\circ_{cu}\), with \(cu\in\mathcal N\).

It remains to identify when two parameters give the same product. If \(\circ_u=\circ_v\), then Theorem~\ref{thm:T-reduction} gives
$
        T_u-T_v\in\mathbb C\operatorname{id}_{\mathfrak g}.
$
Since \(B(u,u)=B(v,v)=0\), both \(T_u\) and \(T_v\) have trace zero, hence this
scalar multiple of the identity must be zero. Thus \(T_u=T_v\). If \(u=0\), then \(v=0\). If \(u\neq0\), then \(\operatorname{Im}T_u=\mathbb C u=\operatorname{Im}T_v\), so \(v=cu\) for some \(c\in\mathbb C^*\). The equality \(T_v=T_u\) gives \(c^2=1\), hence \(v=\pm u\). Therefore \(u\mapsto\circ_u\) induces the asserted bijection from \(\mathcal N/\{\pm1\}\) onto the set of all GD products on \(\mathfrak{sl}_2(\mathbb C)\).
\end{proof}

\begin{cor}
\label{cor:sl2-two-isomorphism-classes}
Up to isomorphism, there are two GD algebra structures on $\mathfrak{sl}_2(\mathbb C)$: the trivial one and the structure defined by 
\[
h \circ f = -2e, \quad f \circ h = 4e, \quad f \circ f = -h.
\]

\end{cor}

\begin{proof}
By Theorem~\ref{thm:sl2-classification}, the GD products on the fixed Lie algebra \(\mathfrak{sl}_2(\mathbb C)\) are parametrized by \(\mathcal N/\{\pm1\}\), where
\(\mathcal N=\{u\in\mathfrak{sl}_2(\mathbb C)\mid B(u,u)=0\}\). The zero product
corresponds to \(u=0\). If \(u\neq0\), then \(u\) is nilpotent, and \(\operatorname{Aut}(\mathfrak{sl}_2(\mathbb C))\) acts transitively on the nonzero nilpotent cone. Hence all nonzero products \(\circ_u\) are isomorphic. Taking \(u=e\), we have
$
        T_e(e)=0,T_e(h)=0$ and $T_e(f)=e.
$
Using \(x\circ_e y=T_e([x,y])-[T_e(x),y]\), we obtain
\[
        h\circ_e f=T_e(-2f)=-2e,\quad
        f\circ_e h=T_e(2f)-[e,h]=2e+2e=4e,\quad
        f\circ_e f=-[e,f]=-h.
\]
This gives the stated representative and proves the corollary.
\end{proof}

\begin{rem}
It is worth comparing this result with the Poisson and transposed Poisson cases. Poisson and transposed Poisson structures on finite-dimensional simple Lie algebras are trivial
\cite{FernandezOuaridi2024,BenayadiBoucetta2014}. By contrast, the preceding classification shows that \(\mathfrak{sl}_2(\mathbb C)\) admits nontrivial GD algebra structures. Thus, on the simple Lie algebra, GD algebra structures are less rigid than both Poisson and transposed Poisson structures.
\end{rem}

\section{The Classification of low-dimensional GD Algebras}
\label{sec:classification-3d}

In this section, we give a complete classification of 2,3-dimensional GD algebras over \(\mathbb C\).

\begin{defn}
Let \((L_1,\circ_1,[\cdot,\cdot]_1)\) and \((L_2,\circ_2,[\cdot,\cdot]_2)\) be two GD algebras. They are said to be isomorphic if there exists a linear isomorphism \(\varphi:L_1\to L_2\) such that
$
        \varphi([x,y]_1)=[\varphi(x),\varphi(y)]_2
$
and
$
        \varphi(x\circ_1 y)=\varphi(x)\circ_2\varphi(y)
$
for all \(x,y\in L_1\).
\end{defn}

\begin{defn}
Let \((L,[\cdot,\cdot])\) be a Lie algebra. Denote by \(\mathcal Z_{\mathrm{GD}}(L)\) the set of all bilinear maps \(\theta:L\times L\to L\) satisfying, for all \(x,y,z\in L\),
\[
        \theta(\theta(x,y),z)-\theta(x,\theta(y,z))
        =
        \theta(\theta(y,x),z)-\theta(y,\theta(x,z)),
\]
\[
        \theta(\theta(x,y),z)=\theta(\theta(x,z),y),
\]
and
\[
        [x,\theta(y,z)]-[z,\theta(y,x)]
        +\theta([y,x],z)-\theta([y,z],x)-\theta(y,[x,z])=0.
\]
\end{defn}

\begin{prop}
\label{prop:fixed-Lie-GD-products}
Let $(L, [\cdot, \cdot])$ be a Lie algebra. For any bilinear map $\theta: L \times L \to L$, define the product $x \circ_\theta y = \theta(x, y)$ for all $x, y \in L$. Then $(L, \circ_\theta, [\cdot, \cdot])$ is a GD algebra if and only if $\theta \in \mathcal{Z}_{\mathrm{GD}}(L)$.
\end{prop}

\begin{proof}
 The proof is straightforward and therefore omitted.
\end{proof}

Consequently, the classification of GD algebra structures on the fixed Lie algebra \((L,[\cdot,\cdot])\) is equivalent to the classification of \(\operatorname{Aut}(L,[\cdot,\cdot])\)-orbits in \(\mathcal Z_{\mathrm{GD}}(L)\).

    \begin{prop}
\label{prop:isomorphism-fixed-Lie}
Let \((L,[\cdot,\cdot])\) be a Lie algebra, and let \(\theta_1,\theta_2\in\mathcal Z_{\mathrm{GD}}(L)\). Then the GD algebras
$
        (L,\circ_{\theta_1},[\cdot,\cdot])
       \text{ and } 
        (L,\circ_{\theta_2},[\cdot,\cdot])
$
are isomorphic if and only if there exists \(\varphi\in\operatorname{Aut}(L [\cdot,\cdot])\) such that
$
        \varphi(\theta_1(x,y))
        =
        \theta_2(\varphi(x),\varphi(y))$ for all $
        x,y \in L.
$
\end{prop}

\begin{proof}
The proof is straightforward and therefore omitted.
\end{proof}

\begin{prop}
\label{prop:transport-GD-structure}
Let \((L_1,[\cdot,\cdot]_1)\) and \((L_2,[\cdot,\cdot]_2)\) be isomorphic Lie algebras. If \(\theta_1\in\mathcal Z_{\mathrm{GD}}(L_1)\), then there exists \(\theta_2\in\mathcal Z_{\mathrm{GD}}(L_2)\) such that the GD algebras
$
        (L_1,\circ_{\theta_1},[\cdot,\cdot]_1)
       \text{ and }
        (L_2,\circ_{\theta_2},[\cdot,\cdot]_2)
$
are isomorphic.
\end{prop}

\begin{proof}
Let \(f:L_1\to L_2\) be a Lie algebra isomorphism. Define \(\theta_2:L_2\times L_2\to L_2\) by
$
        \theta_2(x,y)
        :=
        f\bigl(\theta_1(f^{-1}(x),f^{-1}(y))\bigr),
        \quad \forall x,y \in L_2. 
$
Since \(f\) preserves the Lie brackets, the Novikov identities and the GD compatibility identity for \(\theta_1\) are transported to the corresponding identities for \(\theta_2\). Hence \(\theta_2\in\mathcal Z_{\mathrm{GD}}(L_2)\).
Moreover, for all \(x,y\in L_1\),
$
        f(\theta_1(x,y))
        =
        \theta_2(f(x),f(y)).
$
Thus \(f\) is an isomorphism from \((L_1,\circ_{\theta_1},[\cdot,\cdot]_1)\) to
\((L_2,\circ_{\theta_2},[\cdot,\cdot]_2)\).
\end{proof}

\subsection{The classification of 2-dimensional GD algebras}

Based on the classification method outlined above, we study GD algebra structures on given Lie algebras. Since isomorphic Lie algebras induce isomorphic GD algebra structures, it suffices to consider the classification of complex 2-dimensional Lie algebras up to isomorphism. This classification is given as follows:

\begin{thm}{\rm\cite{Snobl}}
\label{thm:2dim-Lie-classification}
Let \((L,[\cdot,\cdot])\) be a \(2\)-dimensional complex Lie algebra. Then \(L\) is isomorphic to one of the following Lie algebras:
\begin{align*}
        \mathbb C^2:
        &\quad [e_1,e_2]=0, \\
        \mathfrak r_2:
        &\quad [e_1,e_2]=e_1.
\end{align*}

\end{thm}

\begin{thm}\label{thm:2dim-GD-structure}
Let \((L,\circ,[\cdot,\cdot])\) be a complex 2-dimensional GD algebra. Then \((L,\circ,[\cdot,\cdot])\) is isomorphic to one of the following algebras:
\begin{itemize}
    \item \textbf{T1(0):}
    $ \begin{cases} [e_1, e_2] = 0 \end{cases} $

    \item \textbf{T1(1):}
    $ \begin{cases} [e_1, e_2] = e_1 \end{cases} $

    \item \textbf{T2(0):}
    $ \begin{cases} e_2 \circ e_2 = e_1 \end{cases} $

    \item \textbf{T2(1):}
    $ \begin{cases} [e_1, e_2] = e_1 \\ e_2 \circ e_2 = e_1 \end{cases} $

   \item \textbf{T3$_r^{r\in\mathbb{C}}$:}
    $ \begin{cases} [e_1, e_2] = r e_1 \\ e_2 \circ e_1 = -e_1 \end{cases} $

    \item \textbf{N1:}
    $ \begin{cases} e_1 \circ e_1 = e_1, \quad e_2 \circ e_2 = e_2 \end{cases} $

    \item \textbf{N2:}
    $ \begin{cases} e_2 \circ e_2 = e_2 \end{cases} $

    \item \textbf{N3$_r^{r\in\mathbb{C}}$:}
    $ \begin{cases} [e_1, e_2] = r e_1 \\ e_1 \circ e_2 = e_1, \quad e_2 \circ e_1 = e_1, \quad e_2 \circ e_2 = e_2 \end{cases} $

    \item \textbf{N4$_r^{r\in\mathbb{C}}$:}
    $ \begin{cases} [e_1, e_2] = r e_1 \\ e_1 \circ e_2 = e_1, \quad e_2 \circ e_2 = e_2 \end{cases} $

    \item \textbf{N5$_r^{r\in\mathbb{C}}$:}
    $ \begin{cases} [e_1, e_2] = r e_1 \\ e_1 \circ e_2 = e_1, \quad e_2 \circ e_2 = e_1+e_2 \end{cases} $

    \item \textbf{N6$_{\ell,r}^{\ell \in \mathbb{C} \setminus \{0,1\}, r \in \mathbb{C}}$}:
    $ \begin{cases} [e_1, e_2] = r e_1 \\ e_1 \circ e_2 = e_1, \quad e_2 \circ e_1 = \ell e_1, \quad e_2 \circ e_2 = e_2 \end{cases} $
\end{itemize}

Moreover, the parameter equivalences are as follows:
\begin{enumerate}
\item \(\mathbf{T3}_{r_1} \cong \mathbf{T3}_{r_2}\) if and only if \(r_1 = r_2\).
\item \(\mathbf{N3}_{r_1} \cong \mathbf{N3}_{r_2}\) if and only if \(r_1 = r_2\).
\item \(\mathbf{N4}_{r_1} \cong \mathbf{N4}_{r_2}\) if and only if \(r_1 = r_2\).
\item \(\mathbf{N5}_{r_1} \cong \mathbf{N5}_{r_2}\) if and only if \(r_1 = r_2\).
\item \(\mathbf{N6}_{\ell_1,r_1} \cong \mathbf{N6}_{\ell_2,r_2}\) if and only if \(\ell_1 = \ell_2\) and \(r_1 = r_2\).
\end{enumerate}
\end{thm}

\begin{proof}
By Proposition~\ref{prop:fixed-Lie-GD-products}, it suffices to classify the
\(\operatorname{Aut}(L,[\cdot,\cdot])\)-orbits in \(\mathcal Z_{\mathrm{GD}}(L)\) for each 2-dimensional complex Lie algebra \(L\). There are two cases.

For \(L=\mathbb C^2\), the Lie bracket is zero, and hence the GD compatibility identity is identically satisfied. Therefore this branch is exactly the classification of two-dimensional complex Novikov algebras. This gives
\[
        T1(0),\ T2(0),\ T3_0,\ N1,\ N2,\ N3_0,\ N4_0,\ N5_0,\ N6_{\ell,0}.
\]

It remains to consider \(L=\mathfrak r_2\). Choose a basis \(\{e_1,e_2\}\) such that \([e_1,e_2]=e_1\), and set
\[
\begin{aligned}
        e_1\circ e_1&=a e_1+b e_2, &
        e_1\circ e_2&=c e_1+d e_2,&
        e_2\circ e_1&=f e_1+g e_2, &
        e_2\circ e_2&=h e_1+k e_2 .
\end{aligned}
\]
Substitution into the GD compatibility identity gives
\[
        b=0,\qquad d=a,\qquad g=-a,\qquad k=c.
\]
The Novikov identities then force \(a=0\), for instance, right-commutativity for \((e_1,e_1,e_2)\) gives \(2a^2=0\). Thus every compatible product in this branch has the form
\begin{equation}
\label{eq:2d-nonabelian-compatible-product}
        e_1\circ e_2=c e_1,\qquad
        e_2\circ e_1=f e_1,\qquad
        e_2\circ e_2=h e_1+c e_2 .
\end{equation}
Conversely, products of the form
\eqref{eq:2d-nonabelian-compatible-product} satisfy the Novikov identities and the GD compatibility identity.

We now reduce the parameters \(c,f,h\). The automorphisms of the Lie algebra \(L=\mathfrak r_2\) are 
$
        \varphi(e_1)=p e_1 $ and $
        \varphi(e_2)=q e_1+e_2,
        $ where $ p\in\mathbb C^*,\ q\in\mathbb C.
$
If the product with parameters \((c,f,h)\) is transformed into one with parameters \((c',f',h')\), the equations
$
        \varphi(e_i\circ e_j)=\varphi(e_i)\circ'\varphi(e_j)
       $ for $ i,j=1,2
$
give
\[
        c'=c,\qquad f'=f,\qquad h'=ph-qf.
\]
Thus \(c\) and \(f\) are invariant in the fixed non-abelian Lie branch, while \(h\) can be normalized according to whether \(f\) is zero.

If \(f\neq0\), choose \(q=\frac{ph}{f}\), so that \(h'=0\). If \(c=0\), then after the change \(E_1=e_1,\ E_2=-f^{-1}e_2\), we obtain
$
        E_2\circ E_1=-E_1$ and $ [E_1,E_2]=-\frac{1}{f}E_1,
$
which corresponds to \(T3_r\) with \(r=-\frac{1}{f}\). If \(c\neq0\), after the change \(E_1=e_1,\ E_2=c^{-1}e_2\), we obtain
\[
        E_1\circ E_2=E_1,\qquad
        E_2\circ E_1=\frac{f}{c}E_1,\qquad
        E_2\circ E_2=E_2,\qquad
        [E_1,E_2]=\frac{1}{c} E_1.
\]
This gives \(N3_r\) when \(\frac{f}{c}\), and \(N6_{\ell,r}\) when \(\ell=f/c\neq1\), where \(r=\frac{1}{c}\).

Assume next that \(f=0\). Then \(h'=ph\). If \(h=0\), the only nonzero products are
$
        e_1\circ e_2=c e_1$ and $ e_2\circ e_2=c e_2.
$
For \(c=0\), this gives \(T1(1)\). For \(c\neq0\), the change \(E_1=e_1,\ E_2=c^{-1}e_2\) gives
\[
        E_1\circ E_2=E_1,\qquad
        E_2\circ E_2=E_2,\qquad
        [E_1,E_2]=\frac{1}{c}E_1,
\]
which corresponds to \(N4_r\) with \(r=\frac{1}{c}\).

Finally, suppose that \(f=0\) and \(h\neq0\). We normalize \(h=1\). If \(c=0\), then
$
        e_2\circ e_2=e_1$ and $ [e_1,e_2]=e_1,
$
which is \(T2(1)\). If \(c\neq0\), the change \(E_1=c^{-2}e_1,\ E_2=c^{-1}e_2\) gives
\[
        E_1\circ E_2=E_1,\qquad
        E_2\circ E_2=E_1+E_2,\qquad
        [E_1,E_2]=\frac{1}{c}E_1,
\]
which corresponds to \(N5_r\) with \(r=\frac{1}{c}\).

Combining the all algebras listed in the theorem. The same parameter transformation shows that no further identifications occur:
\[
        T3_{r_1}\cong T3_{r_2}
        \Longleftrightarrow r_1=r_2,
\]
\[
        N_i{}_{r_1}\cong N_i{}_{r_2}
        \Longleftrightarrow r_1=r_2,
        \qquad i=3,4,5,
\]
and
\[
        N6_{\ell_1,r_1}\cong N6_{\ell_2,r_2}
        \Longleftrightarrow
        \ell_1=\ell_2,\quad r_1=r_2.
\]
The abelian and non-abelian branches are not isomorphic because their underlying Lie algebras are not isomorphic. 
\end{proof}

Next, we consider the classification of 3-dimensional GD algebras over $\mathbb{C}$. For the 3-dimensional case, solving the defining polynomial identities for a GD product on a fixed Lie algebra usually gives several parameter families. To obtain a classification up to isomorphism, we first simplify the parameter set by isomorphism reduction, and then compare the remaining families by isomorphism testing. We use the
standard reduction and testing methods described in \cite{Zhang}, and therefore do not repeat the details here.

\subsection{The classification of 3-dimensional GD algebras}

Up to isomorphism, the 3-dimensional complex Lie algebras are classified as follows:

\begin{thm}{\rm\cite{Snobl}}
	Let $(L, [\,,\,])$ be a $3$-dimensional nontrivial Lie algebra over the complex field $\mathbb{C}$. Then, up to isomorphism, it is isomorphic to one of the following:
	\begin{align*}
		\mathfrak{h}:           &\quad [e_1, e_2] = e_3, \\
		\mathfrak{g}_1:         &\quad [e_1, e_3] = e_1,\ [e_2, e_3] = e_2, \\
		\mathfrak{g}_2^\alpha:  &\quad [e_1, e_3] = e_1 + e_2,\ [e_2, e_3] = \alpha e_2, \\
		\mathfrak{sl}_2(\mathbb{C}): &\quad [e_1, e_2] = e_3,\ [e_1, e_3] = -e_2,\ [e_2, e_3] = e_1.
	\end{align*}

	Two non-trivial Lie algebras $\mathfrak{g}_2^{\alpha_1}$ and $\mathfrak{g}_2^{\alpha_2}$ are isomorphic if and only if $\alpha_1 = \alpha_2^{-1}$.
		
	\end{thm}

In the following, we provide a complete classification of 3-dimensional GD algebras over the complex field $\mathbb{C}$.
Due to space limitations, we present a typical case to illustrate how the Novikov identities and the GD compatibility equations lead to the classification results.

\begin{thm}
	Let \((L,\circ,[\cdot,\cdot])\) be a complex 3-dimensional GD algebra with nonzero Lie bracket, then \((L,\circ,[\cdot,\cdot])\) is isomorphic to one of the following algebras:

	\begin{itemize}
		\item \textbf{G1$_{\alpha}$:}
		$
		\begin{cases}
			[e_1, e_2] = e_3 \\
			e_1 \circ e_2 = e_2 + \alpha e_3, \quad e_1 \circ e_3 = e_3
		\end{cases}
		$

		\item \textbf{G2$_{\alpha}$:}
		$
		\begin{cases}
			[e_1, e_2] = e_3 \\
			e_1 \circ e_1 = e_3, \quad e_1 \circ e_2 = e_1, \quad e_2 \circ e_1 = e_1 + \alpha e_3, \\
			e_2 \circ e_2 = e_2, \quad e_2 \circ e_3 = e_3, \quad e_3 \circ e_2 = e_3
		\end{cases}
		$

		\item \textbf{G3$_{\alpha,\beta}^{\alpha\neq 0}$:}
		$
		\begin{cases}
			[e_1, e_2] = e_3 \\
			e_1 \circ e_1 = e_1, \quad e_1 \circ e_2 = \alpha e_2+\beta e_3, \quad e_1 \circ e_3 = \alpha e_3, \quad e_2 \circ e_1 = e_2, \quad e_3 \circ e_1 = e_3
		\end{cases}
		$

		\item \textbf{G4$_{\alpha}$:}
		$
		\begin{cases}
			[e_1, e_2] = e_3 \\
			e_1 \circ e_1 = e_1 + e_2, \quad e_1 \circ e_2 = \alpha e_3, \quad e_2 \circ e_1 = e_2, \quad e_3 \circ e_1 = e_3
		\end{cases}
		$

		\item \textbf{G5:}
		$
		\begin{cases}
			[e_1, e_2] = e_3 \\
			e_1 \circ e_1 = e_1 + e_3, \quad e_2 \circ e_1 = e_2, \quad e_3 \circ e_1 = e_3
		\end{cases}
		$

		\item \textbf{G6$_{\alpha}$:}
		$
		\begin{cases}
			[e_1, e_2] = e_3 \\
			e_1 \circ e_2 = -\alpha e_3, \quad e_2 \circ e_1 = \alpha e_3
		\end{cases}
		$

		\item \textbf{G7:}
		$
		\begin{cases}
			[e_1, e_2] = e_3 \\
			e_1 \circ e_1 = e_1, \quad e_1 \circ e_2 = -e_2, \quad e_2 \circ e_1 = e_2
		\end{cases}
		$

		\item \textbf{G8$_{\alpha,\beta}$:}
		$
		\begin{cases}
			[e_1, e_2] = e_3 \\
			e_1 \circ e_2 = \alpha e_3, \quad e_2 \circ e_1 = \beta e_3, \quad e_2 \circ e_2 = e_1
		\end{cases}
		$

		\item \textbf{G9:}
		$
		\begin{cases}
			[e_1, e_2] = e_3 \\
			e_2 \circ e_2 = e_2 + e_3, \quad e_3 \circ e_2 = e_3
		\end{cases}
		$

		\item \textbf{G10:}
		$
		\begin{cases}
			[e_1, e_2] = e_3 \\
			e_2 \circ e_2 = e_2, \quad e_3 \circ e_2 = e_3
		\end{cases}
		$

		\item \textbf{G11$_{\alpha,\beta}$:}
		$
		\begin{cases}
			[e_1, e_2] = e_3 \\
			e_1 \circ e_1 = \alpha e_3, \quad e_1 \circ e_2 = \beta e_3, \quad e_2 \circ e_2 = e_3
		\end{cases}
		$

		\item \textbf{G12$_{\alpha}$:}
		$
		\begin{cases}
			[e_1, e_2] = e_3 \\
			e_1 \circ e_1 = e_1, \quad e_1 \circ e_2 = \alpha e_3, \quad e_2 \circ e_1 = e_2, \quad e_3 \circ e_1 = e_3
		\end{cases}
		$

		\item \textbf{G13$_{\alpha,\beta,\gamma}^{\beta \neq0}$:}
		$
		\begin{cases}
			[e_1, e_3] = e_1, \quad [e_2, e_3] = e_2 \\
			e_1 \circ e_3 = \alpha e_1, \quad e_2 \circ e_3 = \alpha e_2, \quad e_3 \circ e_1 = \beta e_2, \\
			e_3 \circ e_2 = e_1 + \gamma e_2, \quad e_3 \circ e_3 = \alpha e_3
		\end{cases}
		$

		\item \textbf{G14$_{\alpha,\beta}$:}
		$
		\begin{cases}
			[e_1, e_3] = e_1, \quad [e_2, e_3] = e_2 \\
			e_1 \circ e_3 = \alpha e_1, \quad e_2 \circ e_3 = \alpha e_2, \\
			e_3 \circ e_2 = e_1 + \beta e_2, \quad e_3 \circ e_3 = e_2 + \alpha e_3
		\end{cases}
		$

		\item \textbf{G15$_{\alpha,\beta}$:}
		$
		\begin{cases}
			[e_1, e_3] = e_1, \quad [e_2, e_3] = e_2 \\
			e_1 \circ e_3 = \alpha e_1, \quad e_2 \circ e_3 = \alpha e_2, \\
			e_3 \circ e_2 = e_1 + \beta e_2, \quad e_3 \circ e_3 = \alpha e_3
		\end{cases}
		$

		\item \textbf{G16$_{\alpha,\beta}$:}
		$
		\begin{cases}
			[e_1, e_3] = e_1, \quad [e_2, e_3] = e_2 \\
			e_1 \circ e_3 = \alpha e_1, \quad e_2 \circ e_3 = \alpha e_2, \\
			e_3 \circ e_1 = \beta e_1, \quad e_3 \circ e_2 = \beta e_2, \quad e_3 \circ e_3 = \alpha e_3
		\end{cases}
		$

		\item \textbf{G17$_{\alpha}$:}
		$
		\begin{cases}
			[e_1, e_3] = e_1, \quad [e_2, e_3] = e_2 \\
			e_1 \circ e_3 = \alpha e_1, \quad e_2 \circ e_3 = \alpha e_2, \quad e_3 \circ e_3 = e_1 + \alpha e_3
		\end{cases}
		$

		\item \textbf{G18$_{\alpha,\beta,\gamma,\delta}^{\alpha\in\mathbb C\setminus\{0,1,-1\},\ \gamma\delta\ne0}$:}
		$
		\begin{cases}
			[e_1,e_3]=e_1,\quad [e_2,e_3]=\alpha e_2,\\
			e_1\circ e_3=\beta e_1,\quad e_2\circ e_3=\beta e_2,\\
			e_3\circ e_1=\gamma e_1,\quad e_3\circ e_2=\delta e_2,\quad
			e_3\circ e_3=\beta e_3
		\end{cases}
		$

		\item \textbf{G19$_{\alpha,\beta,\gamma,\delta}^{\alpha\in\mathbb C\setminus\{0,1,-1\},\ \beta\in\{0,1\},\ \delta\ne0}$:}
		$
		\begin{cases}
			[e_1,e_3]=e_1,\quad [e_2,e_3]=\alpha e_2,\\
			e_1\circ e_3=\gamma e_1,\quad e_2\circ e_3=\gamma e_2,\\
			e_3\circ e_2=\delta e_2,\quad
			e_3\circ e_3=\beta e_1+\gamma e_3
		\end{cases}
		$

		\item \textbf{G20$_{\alpha,\beta,\gamma,\delta}^{\alpha\in\mathbb C\setminus\{0,1,-1\},\ \beta\in\{0,1\},\ \delta\ne0}$:}
		$
		\begin{cases}
			[e_1,e_3]=e_1,\quad [e_2,e_3]=\alpha e_2,\\
			e_1\circ e_3=\gamma e_1,\quad e_2\circ e_3=\gamma e_2,\\
			e_3\circ e_1=\delta e_1,\quad
			e_3\circ e_3=\beta e_2+\gamma e_3
		\end{cases}
		$

		\item \textbf{G21$_{\alpha,\beta,\gamma,\delta}^{\alpha\in\mathbb C\setminus\{0,1,-1\},\ \beta,\gamma\in\{0,1\}}$:}
		$
		\begin{cases}
			[e_1,e_3]=e_1,\quad [e_2,e_3]=\alpha e_2,\\
			e_1\circ e_3=\delta e_1,\quad e_2\circ e_3=\delta e_2,\quad
			e_3\circ e_3=\beta e_1+\gamma e_2+\delta e_3
		\end{cases}
		$

		\item \textbf{G22$_{\alpha,\beta,\gamma}^{\alpha\in\mathbb C\setminus\{0,1,-1\},\ \beta\in\{0,1\}}$:}
		$
		\begin{cases}
			[e_1,e_3]=e_1,\quad [e_2,e_3]=\alpha e_2,\\
			e_1\circ e_3=\gamma e_1,\quad e_2\circ e_3=e_1+\gamma e_2,\quad
			e_3\circ e_1=-\gamma e_1,\\
			e_3\circ e_2=-\dfrac{\alpha}{\alpha-1}e_1,\quad
			e_3\circ e_3=\beta e_2+\gamma e_3
		\end{cases}
		$

		\item \textbf{G23$_{\alpha,\beta,\gamma}^{\alpha\in\mathbb C\setminus\{0,1,-1\},\ \beta\in\{0,1\}}$:}
		$
		\begin{cases}
			[e_1,e_3]=e_1,\quad [e_2,e_3]=\alpha e_2,\\
			e_1\circ e_3=\gamma e_1+e_2,\quad e_2\circ e_3=\gamma e_2,\quad
			e_3\circ e_1=\dfrac{1}{\alpha-1}e_2,\\
			e_3\circ e_2=-\gamma e_2,\quad
			e_3\circ e_3=\beta e_1+\gamma e_3
		\end{cases}
		$

		\item \textbf{G24$_{\alpha,\beta}^{\alpha\ne 2\beta}$:}
		$
		\begin{cases}
			[e_1,e_3]=e_1,\quad [e_2,e_3]=2e_2,\\
			e_1\circ e_1=e_2,\quad e_1\circ e_3=\alpha e_1,\quad e_2\circ e_3=\alpha e_2,\quad
			e_3\circ e_1=\beta e_1,\\
            e_3\circ e_2=(-\alpha+2\beta)e_2,\quad
			e_3\circ e_3=\alpha e_3
		\end{cases}
		$

		\item \textbf{G25$_{\alpha,\beta}^{\alpha\in\{0,1\}}$:}
		$
		\begin{cases}
			[e_1,e_3]=e_1,\quad [e_2,e_3]=2e_2,\\
			e_1\circ e_1=e_2,\quad e_1\circ e_3=2\beta e_1,\quad
			e_2\circ e_3=2\beta e_2,\quad e_3\circ e_1=\beta e_1,\\
			e_3\circ e_3=\alpha e_2+2\beta e_3
		\end{cases}
		$

		\item \textbf{G26$_{\alpha,\beta}^{\beta\ne 0}$:}
		$
		\begin{cases}
			[e_1,e_3]=e_1,\quad [e_2,e_3]=\dfrac12 e_2,\\
			e_2\circ e_2=e_1,\quad e_1\circ e_3=\alpha e_1,\quad e_2\circ e_3=\alpha e_2,\quad
			e_3\circ e_1=\beta e_1,\\
            e_3\circ e_2=\dfrac{\alpha+\beta}{2}e_2,\quad
			e_3\circ e_3=\alpha e_3
		\end{cases}
		$

		\item \textbf{G27$_{\alpha,\beta}^{\alpha\in\{0,1\}}$:}
		$
		\begin{cases}
			[e_1,e_3]=e_1,\quad [e_2,e_3]=\dfrac12 e_2,\\
			e_2\circ e_2=e_1,\quad e_1\circ e_3=\beta e_1,\quad e_2\circ e_3=\beta e_2,\quad
			e_3\circ e_2=\dfrac \beta 2 e_2,\\
			e_3\circ e_3=\alpha e_1+\beta e_3
		\end{cases}
		$

		\item \textbf{G28$_{\alpha,\beta,\gamma}^{\beta\gamma\ne 0}$:}
		$
		\begin{cases}
			[e_1,e_3]=e_1,\quad [e_2,e_3]=-e_2,\\
			e_1\circ e_3=\alpha e_1,\quad e_2\circ e_3=\alpha e_2,\quad
			e_3\circ e_1=\beta e_1,\quad e_3\circ e_2=\gamma e_2,\quad
			e_3\circ e_3=\alpha e_3
		\end{cases}
		$

		\item \textbf{G29$_{\alpha,\beta,\gamma}^{\alpha\in\{0,1\},\ \gamma\ne 0}$:}
		$
		\begin{cases}
			[e_1,e_3]=e_1,\quad [e_2,e_3]=-e_2,\\
			e_1\circ e_3=\beta e_1,\quad e_2\circ e_3=\beta e_2,\quad
			e_3\circ e_2=\gamma e_2,\quad
			e_3\circ e_3=\alpha e_1+\beta e_3
		\end{cases}
		$

		\item \textbf{G30$_{\alpha}$:}
		$
		\begin{cases}
			[e_1,e_3]=e_1,\quad [e_2,e_3]=-e_2,\\
			e_1\circ e_3=\alpha e_1,\quad e_2\circ e_3=\alpha e_2,\quad e_3\circ e_3=\alpha e_3
		\end{cases}
		$

		\item \textbf{G31$_{\alpha}$:}
		$
		\begin{cases}
			[e_1,e_3]=e_1,\quad [e_2,e_3]=-e_2,\\
			e_1\circ e_3=\alpha e_1,\quad e_2\circ e_3=\alpha e_2,\quad
			e_3\circ e_3=e_1+e_2+\alpha e_3
		\end{cases}
		$

		\item \textbf{G32$_{\alpha}$:}
		$
		\begin{cases}
			[e_1,e_3]=e_1,\quad [e_2,e_3]=-e_2,\\
			e_1\circ e_3=\alpha e_1,\quad e_2\circ e_3=\alpha e_2,\quad
			e_3\circ e_3=e_1+\alpha e_3
		\end{cases}
		$

		\item \textbf{G33$_{\alpha,\beta}^{\alpha\in\{0,1\}}$:}
		$
		\begin{cases}
			[e_1,e_3]=e_1,\quad [e_2,e_3]=-e_2,\\
			e_1\circ e_3=\beta e_1,\quad e_2\circ e_3=e_1+\beta e_2,\quad
			e_3\circ e_1=-\beta e_1,\\
            e_3\circ e_2=-\dfrac12 e_1,\quad
			e_3\circ e_3=\alpha e_2+\beta e_3
		\end{cases}
		$

		\item \textbf{G34$_{\alpha,\beta,\gamma}^{\beta\ne 0}$:}
		$
		\begin{cases}
			[e_1,e_3]=e_1+e_2,\quad [e_2,e_3]=e_2,\\
			e_1\circ e_3=\alpha e_1,\quad e_2\circ e_3=\alpha e_2,\quad
			e_3\circ e_1=\beta e_1+\gamma e_2,\\
            e_3\circ e_2=\beta e_2,\quad
			e_3\circ e_3=\alpha e_3
		\end{cases}
		$

		\item \textbf{G35$_{\alpha,\beta,\gamma}^{\alpha\in\{0,1\},\ \gamma\ne 0}$:}
		$
		\begin{cases}
			[e_1,e_3]=e_1+e_2,\quad [e_2,e_3]=e_2,\\
			e_1\circ e_3=\beta e_1,\quad e_2\circ e_3=\beta e_2,\quad
			e_3\circ e_1=\gamma e_2,\quad
			e_3\circ e_3=\alpha e_1+\beta e_3
		\end{cases}
		$

		\item \textbf{G36$_{\alpha}$:}
		$
		\begin{cases}
			[e_1,e_3]=e_1+e_2,\quad [e_2,e_3]=e_2,\\
			e_1\circ e_3=\alpha e_1,\quad e_2\circ e_3=\alpha e_2,\quad e_3\circ e_3=\alpha e_3
		\end{cases}
		$

		\item \textbf{G37$_{\alpha}$:}
		$
		\begin{cases}
			[e_1,e_3]=e_1+e_2,\quad [e_2,e_3]=e_2,\\
			e_1\circ e_3=\alpha e_1,\quad e_2\circ e_3=\alpha e_2,\quad
			e_3\circ e_3=e_1+\alpha e_3
		\end{cases}
		$

		\item \textbf{G38$_{\alpha}$:}
		$
		\begin{cases}
			[e_1,e_3]=e_1+e_2,\quad [e_2,e_3]=e_2,\\
			e_1\circ e_3=\alpha e_1,\quad e_2\circ e_3=\alpha e_2,\quad
			e_3\circ e_3=e_2+\alpha e_3
		\end{cases}
		$

		\item \textbf{G39$_{\alpha}^{\alpha\ne 0}$:}
		$
		\begin{cases}
			[e_1,e_3]=e_1+e_2,\quad [e_2,e_3]=e_2,\\
			e_1\circ e_3=\alpha(e_1+e_2),\quad e_2\circ e_3=\alpha e_2,\quad
			e_3\circ e_2=-\alpha e_2,\quad e_3\circ e_3=\alpha e_3
		\end{cases}
		$

		\item \textbf{G40$_{\alpha}^{\alpha\ne 0}$:}
		$
		\begin{cases}
			[e_1,e_3]=e_1+e_2,\quad [e_2,e_3]=e_2,\\
			e_1\circ e_3=\alpha(e_1+e_2),\quad e_2\circ e_3=\alpha e_2,\quad
			e_3\circ e_2=-\alpha e_2,\quad e_3\circ e_3=e_1+\alpha e_3
		\end{cases}
		$

		\item \textbf{G41$_{\alpha}^{\alpha\in\mathbb C\setminus\{0,-1\}}$:}
		$
		\begin{cases}
			[e_1,e_3]=e_1,\\
			e_1\circ e_2=e_1,\quad e_2\circ e_1=\alpha e_1,\quad
			e_2\circ e_2=e_2,\quad e_3\circ e_2=e_3
		\end{cases}
		$

		\item \textbf{G42$_{\alpha}^{\alpha\in\{0,1\}}$:}
		$
		\begin{cases}
			[e_1,e_3]=e_1,\\
			e_1\circ e_2=e_1,\quad e_2\circ e_2=e_2,\quad
			e_2\circ e_3=\alpha e_1,\quad e_3\circ e_2=e_3
		\end{cases}
		$

		\item \textbf{G43$_{\alpha}^{\alpha\in\{0,1\}}$:}
		$
		\begin{cases}
		[e_1,e_3]=e_1,\\
			e_1\circ e_2=e_1,\quad e_2\circ e_1=-e_1,\quad
			e_2\circ e_2=e_2,\quad e_3\circ e_2=e_3,\quad
			e_3\circ e_3=\alpha e_1
		\end{cases}
		$

		\item \textbf{G44:}
		$
		\begin{cases}
			[e_1,e_3]=e_1,\\
			e_1\circ e_3=e_2,\quad e_3\circ e_1=-e_2
		\end{cases}
		$

		\item \textbf{G45$_{\alpha}$:}
		$
		\begin{cases}
			[e_1,e_3]=e_1,\\
			e_1\circ e_3=e_2,\quad e_3\circ e_1=-e_2,\quad
			e_3\circ e_3=e_1+\alpha e_2
		\end{cases}
		$

		\item \textbf{G46:}
		$
		\begin{cases}
			[e_1,e_3]=e_1,\\
			e_1\circ e_3=e_2,\quad e_3\circ e_1=-e_2,\quad e_3\circ e_3=e_2
		\end{cases}
		$

		\item \textbf{G47$_{\alpha}^{\alpha\ne 0}$:}
		$
		\begin{cases}
			[e_1,e_3]=e_1,\\
			e_1\circ e_3=\alpha e_1+e_2,\quad
			e_3\circ e_1=-\alpha e_1-e_2,\quad e_3\circ e_3=\alpha e_3
		\end{cases}
		$

		\item \textbf{G48$_{\alpha,\beta}^{\alpha\in\{0,1\},\ \beta\ne 0}$:}
		$
		\begin{cases}
			[e_1,e_3]=e_1,\\
			e_1\circ e_3=\beta e_1+e_2,\quad e_2\circ e_3=\beta e_2,\quad
			e_3\circ e_1=-e_2,\\
            e_3\circ e_2=-\beta e_2,\quad
			e_3\circ e_3=\alpha e_1+\beta e_3
		\end{cases}
		$

		\item \textbf{G49:}
		$
		\begin{cases}
			[e_1,e_3]=e_1,\\
			e_2\circ e_1=e_1
		\end{cases}
		$

		\item \textbf{G50$_{\alpha,\beta}^{\beta\ne 0}$:}
		$
		\begin{cases}
			[e_1,e_3]=e_1,\\
			e_1\circ e_3=\alpha e_1,\quad e_2\circ e_2=e_2,\quad
			e_3\circ e_1=\beta e_1,\quad e_3\circ e_3=\alpha e_3
		\end{cases}
		$

		\item \textbf{G51$_{\alpha,\beta}^{\alpha\in\{0,1\}}$:}
		$
		\begin{cases}
			[e_1,e_3]=e_1,\\
			e_1\circ e_3=\beta e_1,\quad e_2\circ e_2=e_2,\quad
			e_3\circ e_3=\alpha e_1+\beta e_3
		\end{cases}
		$

		\item \textbf{G52$_{\alpha,\beta,\gamma}^{\beta\gamma\ne 0}$:}
		$
		\begin{cases}
			[e_1,e_3]=e_1,\\
			e_1\circ e_3=\alpha e_1,\quad e_2\circ e_3=\alpha e_2,\quad
			e_3\circ e_1=\beta e_1,\quad e_3\circ e_2=\gamma e_2,\quad
			e_3\circ e_3=\alpha e_3
		\end{cases}
		$

		\item \textbf{G53$_{\alpha,\beta,\gamma}^{\alpha\in\{0,1\},\ \gamma\ne 0}$:}
		$
		\begin{cases}
			[e_1,e_3]=e_1,\\
			e_1\circ e_3=\beta e_1,\quad e_2\circ e_3=\beta e_2,\quad
			e_3\circ e_1=\gamma e_1,\quad e_3\circ e_3=\alpha e_2+\beta e_3
		\end{cases}
		$

		\item \textbf{G54$_{\alpha,\beta,\gamma}^{\alpha\in\{0,1\},\ \gamma\ne 0}$:}
		$
		\begin{cases}
			[e_1,e_3]=e_1,\\
			e_1\circ e_3=\beta e_1,\quad e_2\circ e_3=\beta e_2,\quad
			e_3\circ e_2=\gamma e_2,\quad e_3\circ e_3=\alpha e_1+\beta e_3
		\end{cases}
		$

		\item \textbf{G55$_{\alpha,\beta,\gamma}^{\alpha,\beta\in\{0,1\}}$:}
		$
		\begin{cases}
			[e_1,e_3]=e_1,\\
			e_1\circ e_3=\gamma e_1,\quad e_2\circ e_3=\gamma e_2,\quad
			e_3\circ e_3=\alpha e_1+\beta e_2+\gamma e_3
		\end{cases}
		$

		\item \textbf{G56$_{\alpha,\beta}^{\alpha\beta\ne 0}$:}
		$
		\begin{cases}
			[e_1,e_3]=e_1,\\
			e_1\circ e_3=\alpha e_1,\quad e_3\circ e_1=\beta e_1,\quad e_3\circ e_3=\alpha e_3
		\end{cases}
		$

		\item \textbf{G57$_{\alpha,\beta}^{\alpha\in\{0,1\},\ \beta\ne 0}$:}
		$
		\begin{cases}
		[e_1,e_3]=e_1,\\
			e_1\circ e_3=\beta e_1,\quad e_3\circ e_3=\alpha e_1+\beta e_3
		\end{cases}
		$

		\item \textbf{G58$_{\alpha,\beta}^{\alpha\in\{0,1\}}$:}
		$
		\begin{cases}
			[e_1,e_3]=e_1,\\
			e_1\circ e_3=\beta e_1,\quad e_2\circ e_3=e_1,\quad
			e_3\circ e_3=\alpha e_2+\beta e_3
		\end{cases}
		$

		\item \textbf{G59$_{\alpha,\beta}^{\alpha\in\{0,1\},\ \beta\ne 0}$:}
		$
		\begin{cases}
			[e_1,e_3]=e_1,\\
			e_1\circ e_3=\beta e_1,\quad e_2\circ e_3=e_1+\beta e_2,\quad
			e_3\circ e_1=-\beta e_1,\\
            e_3\circ e_3=\alpha e_2+\beta e_3
		\end{cases}
        $
        
        \item \textbf{G60:}
$
\begin{cases}
	[e_1, e_2] = e_3, \quad [e_1, e_3] = -e_2, \quad [e_2, e_3] = e_1
\end{cases}
$

\item \textbf{G61:}
$
\begin{cases}
	[e_1, e_2] = e_3, \quad [e_1, e_3] = -e_2, \quad [e_2, e_3] = e_1, \\[1ex]
	e_2 \circ e_2 = -\dfrac{1}{4} e_3, \quad e_2 \circ e_3 = e_1, \quad e_3 \circ e_2 = -\dfrac{1}{2} e_1
\end{cases}
$

	\end{itemize}

	\noindent

\noindent
$\mathbf{G1}_{\alpha_1} \cong \mathbf{G1}_{\alpha_2}$ if and only if $\alpha_1 = \alpha_2$. \\
$\mathbf{G2}_{\alpha_1} \cong \mathbf{G2}_{\alpha_2}$ if and only if $\alpha_1 = \alpha_2$. \\
$\mathbf{G3}_{\alpha_1, \beta_1} \cong \mathbf{G3}_{\alpha_2, \beta_2}$ if and only if $(\alpha_1, \beta_1) = (\alpha_2, \beta_2)$. \\
$\mathbf{G4}_{\alpha_1} \cong \mathbf{G4}_{\alpha_2}$ if and only if $\alpha_1 = \alpha_2$. \\
$\mathbf{G6}_{\alpha_1} \cong \mathbf{G6}_{\alpha_2}$ if and only if $\alpha_1 = \alpha_2$. \\
$\mathbf{G8}_{\alpha_1, \beta_1} \cong \mathbf{G8}_{\alpha_2, \beta_2}$ if and only if $(\alpha_1, \beta_1) = (\alpha_2, \beta_2)$. \\
$\mathbf{G11}_{\alpha_1, \beta_1} \cong \mathbf{G11}_{\alpha_2, \beta_2}$ if and only if $(\alpha_1, \beta_1) = (\alpha_2, \beta_2)$. \\
$\mathbf{G12}_{\alpha_1} \cong \mathbf{G12}_{\alpha_2}$ if and only if $\alpha_1 = \alpha_2$. \\
$\mathbf{G13}_{\alpha_1, \beta_1, \gamma_1} \cong \mathbf{G13}_{\alpha_2, \beta_2, \gamma_2}$ if and only if $(\alpha_1, \beta_1, \gamma_1) = (\alpha_2, \beta_2, \gamma_2)$. \\
$\mathbf{G14}_{\alpha_1, \beta_1} \cong \mathbf{G14}_{\alpha_2, \beta_2}$ if and only if $(\alpha_1, \beta_1) = (\alpha_2, \beta_2)$. \\
$\mathbf{G15}_{\alpha_1, \beta_1} \cong \mathbf{G15}_{\alpha_2, \beta_2}$ if and only if $(\alpha_1, \beta_1) = (\alpha_2, \beta_2)$. \\
$\mathbf{G16}_{\alpha_1, \beta_1} \cong \mathbf{G16}_{\alpha_2, \beta_2}$ if and only if $(\alpha_1, \beta_1) = (\alpha_2, \beta_2)$. \\
$\mathbf{G17}_{\alpha_1} \cong \mathbf{G17}_{\alpha_2}$ if and only if $\alpha_1 = \alpha_2$. \\ 
 $\mathbf{G18}_{\alpha_1, \beta_1, \gamma_1, \delta_1} \cong \mathbf{G18}_{\alpha_2, \beta_2, \gamma_2, \delta_2}$ if and only if $(\alpha_1, \beta_1, \gamma_1, \delta_1) = (\alpha_2^{-1}, \beta_2, \gamma_2, \delta_2)$.\\
$\mathbf{G19}_{\alpha_1, \beta_1, \gamma_1, \delta_1} \cong \mathbf{G19}_{\alpha_2, \beta_2, \gamma_2, \delta_2}$ if and only if $(\alpha_1, \beta_1, \gamma_1, \delta_1) = (\alpha_2^{-1}, \beta_2, \gamma_2, \delta_2)$.\\
$\mathbf{G20}_{\alpha_1, \beta_1, \gamma_1, \delta_1} \cong \mathbf{G20}_{\alpha_2, \beta_2, \gamma_2, \delta_2}$ if and only if $(\alpha_1, \beta_1, \gamma_1, \delta_1) = (\alpha_2^{-1}, \beta_2, \gamma_2, \delta_2)$.\\
$\mathbf{G21}_{\alpha_1, \beta_1, \gamma_1, \delta_1} \cong \mathbf{G21}_{\alpha_2, \beta_2, \gamma_2, \delta_2}$ if and only if $(\alpha_1, \beta_1, \gamma_1, \delta_1) = (\alpha_2^{-1}, \beta_2, \gamma_2, \delta_2)$.\\
$\mathbf{G22}_{\alpha_1, \beta_1, \gamma_1} \cong \mathbf{G22}_{\alpha_2, \beta_2, \gamma_2}$ if and only if $(\alpha_1, \beta_1, \gamma_1) = (\alpha_2^{-1}, \beta_2, \gamma_2)$.\\
$\mathbf{G23}_{\alpha_1, \beta_1, \gamma_1} \cong \mathbf{G23}_{\alpha_2, \beta_2, \gamma_2}$ if and only if $(\alpha_1, \beta_1, \gamma_1) = (\alpha_2^{-1}, \beta_2, \gamma_2)$.\\
	$\mathbf{G24}_{\alpha_1, \beta_1} \cong \mathbf{G24}_{\alpha_2, \beta_2}$ if and only if $(\alpha_1, \beta_1) = (\alpha_2, \beta_2)$. \\
	$\mathbf{G25}_{\alpha_1, \beta_1} \cong \mathbf{G25}_{\alpha_2, \beta_2}$ if and only if $(\alpha_1, \beta_1) = (\alpha_2, \beta_2)$. \\
	$\mathbf{G26}_{\alpha_1, \beta_1} \cong \mathbf{G26}_{\alpha_2, \beta_2}$ if and only if $(\alpha_1, \beta_1) = (\alpha_2, \beta_2)$. \\
	$\mathbf{G27}_{\alpha_1, \beta_1} \cong \mathbf{G27}_{\alpha_2, \beta_2}$ if and only if $(\alpha_1, \beta_1) = (\alpha_2, \beta_2)$. \\
	$\mathbf{G28}_{\alpha_1, \beta_1, \gamma_1} \cong \mathbf{G28}_{\alpha_2, \beta_2, \gamma_2}$ if and only if $(\alpha_1, \beta_1, \gamma_1) = (\alpha_2, \beta_2, \gamma_2)$ or $(\alpha_1, \beta_1, \gamma_1) = (-\alpha_2, -\gamma_2,-\beta_2, )$. \\
	$\mathbf{G29}_{\alpha_1, \beta_1, \gamma_1} \cong \mathbf{G29}_{\alpha_2, \beta_2, \gamma_2}$ if and only if $(\alpha_1, \beta_1, \gamma_1) = (\alpha_2, \beta_2, \gamma_2)$. \\
	$\mathbf{G30}_{\alpha_1} \cong \mathbf{G30}_{\alpha_2}$ if and only if $\alpha_1 = \pm\alpha_2$. \\
	$\mathbf{G31}_{\alpha_1} \cong \mathbf{G31}_{\alpha_2}$ if and only if $\alpha_1 = \pm\alpha_2$. \\
	$\mathbf{G32}_{\alpha_1} \cong \mathbf{G32}_{\alpha_2}$ if and only if $\alpha_1 = \alpha_2$. \\
	$\mathbf{G33}_{\alpha_1, \beta_1} \cong \mathbf{G33}_{\alpha_2, \beta_2}$ if and only if $(\alpha_1, \beta_1) = (\alpha_2, \beta_2)$. \\
	$\mathbf{G34}_{\alpha_1, \beta_1, \gamma_1} \cong \mathbf{G34}_{\alpha_2, \beta_2, \gamma_2}$ if and only if $(\alpha_1, \beta_1, \gamma_1) = (\alpha_2, \beta_2, \gamma_2)$. \\
	$\mathbf{G35}_{\alpha_1, \beta_1, \gamma_1} \cong \mathbf{G35}_{\alpha_2, \beta_2, \gamma_2}$ if and only if $(\alpha_1, \beta_1, \gamma_1) = (\alpha_2, \beta_2, \gamma_2)$. \\
	$\mathbf{G36}_{\alpha_1} \cong \mathbf{G36}_{\alpha_2}$ if and only if $\alpha_1 = \alpha_2$. \\
	$\mathbf{G37}_{\alpha_1} \cong \mathbf{G37}_{\alpha_2}$ if and only if $\alpha_1 = \alpha_2$. \\
	$\mathbf{G38}_{\alpha_1} \cong \mathbf{G38}_{\alpha_2}$ if and only if $\alpha_1 = \alpha_2$. \\
	$\mathbf{G39}_{\alpha_1} \cong \mathbf{G39}_{\alpha_2}$ if and only if $\alpha_1 = \alpha_2$. \\
	$\mathbf{G40}_{\alpha_1} \cong \mathbf{G40}_{\alpha_2}$ if and only if $\alpha_1 = \alpha_2$. \\
	$\mathbf{G41}_{\alpha_1} \cong \mathbf{G41}_{\alpha_2}$ if and only if $\alpha_1 = \alpha_2$. \\
	$\mathbf{G42}_{\alpha_1} \cong \mathbf{G42}_{\alpha_2}$ if and only if $\alpha_1 = \alpha_2$. \\
	$\mathbf{G43}_{\alpha_1} \cong \mathbf{G43}_{\alpha_2}$ if and only if $\alpha_1 = \alpha_2$. \\
	$\mathbf{G45}_{\alpha_1} \cong \mathbf{G45}_{\alpha_2}$ if and only if $\alpha_1 = \alpha_2$. \\
	$\mathbf{G47}_{\alpha_1} \cong \mathbf{G47}_{\alpha_2}$ if and only if $\alpha_1 = \alpha_2$. \\
	$\mathbf{G48}_{\alpha_1, \beta_1} \cong \mathbf{G48}_{\alpha_2, \beta_2}$ if and only if $(\alpha_1, \beta_1) = (\alpha_2, \beta_2)$. \\
	$\mathbf{G50}_{\alpha_1, \beta_1} \cong \mathbf{G50}_{\alpha_2, \beta_2}$ if and only if $(\alpha_1, \beta_1) = (\alpha_2, \beta_2)$. \\
	$\mathbf{G51}_{\alpha_1, \beta_1} \cong \mathbf{G51}_{\alpha_2, \beta_2}$ if and only if $(\alpha_1, \beta_1) = (\alpha_2, \beta_2)$. \\
	$\mathbf{G52}_{\alpha_1, \beta_1, \gamma_1} \cong \mathbf{G52}_{\alpha_2, \beta_2, \gamma_2}$ if and only if $(\alpha_1, \beta_1, \gamma_1) = (\alpha_2, \beta_2, \gamma_2)$. \\
	$\mathbf{G53}_{\alpha_1, \beta_1, \gamma_1} \cong \mathbf{G53}_{\alpha_2, \beta_2, \gamma_2}$ if and only if $(\alpha_1, \beta_1, \gamma_1) = (\alpha_2, \beta_2, \gamma_2)$. \\
	$\mathbf{G54}_{\alpha_1, \beta_1, \gamma_1} \cong \mathbf{G54}_{\alpha_2, \beta_2, \gamma_2}$ if and only if $(\alpha_1, \beta_1, \gamma_1) = (\alpha_2, \beta_2, \gamma_2)$. \\
	$\mathbf{G55}_{\alpha_1, \beta_1, \gamma_1} \cong \mathbf{G55}_{\alpha_2, \beta_2, \gamma_2}$ if and only if $(\alpha_1, \beta_1, \gamma_1) = (\alpha_2, \beta_2, \gamma_2)$. \\
	$\mathbf{G56}_{\alpha_1, \beta_1} \cong \mathbf{G56}_{\alpha_2, \beta_2}$ if and only if $(\alpha_1, \beta_1) = (\alpha_2, \beta_2)$. \\
	$\mathbf{G57}_{\alpha_1, \beta_1} \cong \mathbf{G57}_{\alpha_2, \beta_2}$ if and only if $(\alpha_1, \beta_1) = (\alpha_2, \beta_2)$. \\
	$\mathbf{G58}_{\alpha_1, \beta_1} \cong \mathbf{G58}_{\alpha_2, \beta_2}$ if and only if $(\alpha_1, \beta_1) = (\alpha_2, \beta_2)$. \\
	$\mathbf{G59}_{\alpha_1, \beta_1} \cong \mathbf{G59}_{\alpha_2, \beta_2}$ if and only if $(\alpha_1, \beta_1) = (\alpha_2, \beta_2)$. \\
\end{thm}
 
\begin{proof}
We explain one representative branch
\(G18_{\alpha,\beta,\gamma,\delta}^{\alpha\in\mathbb C\setminus\{0,1,-1\},\
\gamma\delta\neq0}\), to illustrate how the normal forms are obtained from the Novikov identities and the GD compatibility identity.

Fix the 3-dimensional Lie algebra
$
        [e_1,e_3]=e_1$ and $ [e_2,e_3]=\alpha e_2,
        $ where $ \alpha\in\mathbb C\setminus\{0,1,-1\}.
$
By Proposition~\ref{prop:fixed-Lie-GD-products}, GD products on this fixed Lie algebra are obtained by imposing the Novikov identities and the GD compatibility identity on the bilinear map \(\theta(x,y)=x\circ y\).

We first use the GD compatibility identity. Write
\[
\begin{aligned}
        e_3\circ e_1&=\gamma e_1+b e_2+c e_3,&
        e_3\circ e_2&=d e_1+\delta e_2+f e_3,&
        e_3\circ e_3&=p e_1+q e_2+\beta e_3 .
\end{aligned}
\]
Substituting these expressions into the GD compatibility identity determines the remaining products:
\begin{equation}
\label{eq:G18-preliminary-products}
\begin{aligned}
        e_1\circ e_1&=-c e_1,
        &
        e_1\circ e_2&=-\alpha c e_2,
        &
        e_1\circ e_3&=\beta e_1+(\alpha-1)b e_2-c e_3,\\
        e_2\circ e_1&=-\frac{f}{\alpha}e_1,
        &
        e_2\circ e_2&=-f e_2,
        &
        e_2\circ e_3&=\frac{1-\alpha}{\alpha}d e_1+\beta e_2-f e_3 .
\end{aligned}
\end{equation}
Thus, after imposing the GD compatibility identity, the possible products in this branch are controlled by the parameters
$
        \beta,\gamma,\delta,b,c,d,f,p,q .
$

The branch \(G18\) is the case
$
        \gamma\delta\neq0.
$
Now substitute \eqref{eq:G18-preliminary-products} into the two Novikov identities. Among the resulting equations, we have
\[
        2c^2=0,\quad
        2f^2=0,\quad
        (\alpha-1)b\gamma=0,\quad
        \frac{\alpha-1}{\alpha}d\delta=0.
\]
Since \(\alpha\notin\{0,1\}\) and \(\gamma\delta\neq0\), it follows that
$
        c=f=b=d=0.
$
Consequently the product reduces to
\begin{equation}
\label{eq:G18-before-normalization}
\begin{aligned}
    e_1\circ e_3 &= \beta e_1, \quad & e_2\circ e_3 &= \beta e_2, \quad & e_3\circ e_3 &= p e_1 + q e_2 + \beta e_3.\\
    e_3\circ e_1 &= \gamma e_1, \quad & e_3\circ e_2 &= \delta e_2.
\end{aligned}
\end{equation}

It remains to remove the inessential parameters \(p\) and \(q\). For
\(\alpha\in\mathbb C\setminus\{0,1,-1\}\), the following transformations are automorphisms of the fixed Lie algebra:
\[
        e_1\mapsto a e_1,\quad
        e_2\mapsto d_0 e_2,\quad
        e_3\mapsto e_3+s e_1+t e_2,
        \quad a d_0\neq0.
\]
Under the transformation \(e_3 \mapsto e_3 + s e_1 + t e_2\), the coefficients of \(e_1\) and \(e_2\) in the expression of \(e_3 \circ e_3\) from \eqref{eq:G18-before-normalization} transform into \(p+\gamma s\) and \(q+\delta t\), respectively. Since \(\gamma \delta \neq 0\), we may set \(s=-\frac{p}{\gamma}\) and \(t=-\frac{q}{\delta}\) to achieve the normalization \(p=q=0\).

Therefore we obtain the normal form
\[
\begin{cases}
        [e_1,e_3]=e_1,\quad [e_2,e_3]=\alpha e_2,\\
        e_1\circ e_3=\beta e_1,\quad e_2\circ e_3=\beta e_2,\quad
        e_3\circ e_1=\gamma e_1,\quad e_3\circ e_2=\delta e_2,\quad
        e_3\circ e_3=\beta e_3.
\end{cases}
\]
This is exactly the family
$
        G18_{\alpha,\beta,\gamma,\delta}^{
        \alpha\in\mathbb C\setminus\{0,1,-1\},\ \gamma\delta\neq0}.
$

Finally, after the normalization \(p=q=0\), the remaining diagonal automorphisms \(e_1\mapsto a e_1\), \(e_2\mapsto d_0 e_2\) preserve \(\beta,\gamma,\delta\). Hence, for fixed generic \(\alpha\), no further parameter reduction occurs in this branch.

The other branches are obtained in the same way: one first imposes the GD compatibility identity, then the Novikov identities, and finally reduces the remaining parameters by automorphisms of the fixed Lie algebra. This gives exactly the normal forms listed in the theorem.
\end{proof}

In what follows, we determine which of the above algebras are special. Since it was proved in \cite{Kolesnikov-Sartayev} that every 2-dimensional GD algebra is special, we now consider the 3-dimensional case.

\begin{longtable}{@{} p{3.5cm} >{\raggedright\arraybackslash}p{6.5cm} >{\raggedright\arraybackslash}p{4.5cm} @{}}
\caption{Summary of special and non-special GD algebra structures.}
\label{tab:special_summary} \\
\toprule
\textbf{Lie Algebra} & \textbf{Special Cases} & \textbf{Non-Special Cases} \\
\midrule
\endfirsthead

\multicolumn{3}{c}%
{{\bfseries \tablename\ \thetable{} -- continued from previous page}} \\
\toprule
\textbf{Lie Algebra} & \textbf{Special Cases} & \textbf{Non-Special Cases} \\
\midrule
\endhead

\midrule \multicolumn{3}{r}{{Continued on next page}} \\
\endfoot

\bottomrule
\endlastfoot

$\mathfrak{h}$ 
& $G1_\alpha$, $G2_\alpha$, $G3_{\alpha,\beta}^{\alpha\neq 0}$, \par $G4_\alpha$, $G5$, $G6_\alpha$, \par $G8_{\alpha,\beta}$, $G11_{\alpha,\beta}$, $G12_\alpha$
& $G7$, $G9$, $G10$ \\
\midrule

$\mathfrak{g}_1$ 
& $G13_{\alpha,\beta,\gamma}^{\beta\neq 0}$, $G14_{\alpha,\beta}$, $G15_{\alpha,\beta}$, \par $G16_{\alpha,\beta}$, $G17_{\alpha}$ 
& None \\
\midrule

$\mathfrak{g}_2^{\alpha}$ 
\newline {\footnotesize $(\alpha \notin \{0, 1, -1\})$}
& $G18_{\alpha,\beta,\gamma,\delta}$, $G19_{\alpha,\beta,\gamma,\delta}$, \par $G20_{\alpha,\beta,\gamma,\delta}$, $G21_{\alpha,\beta,\gamma,\delta}$, \par $G22_{\alpha,\beta,\gamma}$ ($\beta \neq 1$ or $\alpha = 1/2$), \par $G23_{\alpha,\beta,\gamma}$ ($\beta \neq 1$ or $\alpha = 2$)
& $G22_{\alpha,\beta,\gamma}$ ($\beta = 1$, $\alpha \neq 1/2$), \par $G23_{\alpha,\beta,\gamma}$ ($\beta = 1$, $\alpha \neq 2$) \\
\midrule

$\mathfrak{g}_2^{2}$ 
& $G24_{a,b}^{a\neq 2b}$, $G25_{\varepsilon,b}^{\varepsilon\in\{0,1\}}$
& None \\
\midrule

$\mathfrak{g}_2^{1/2}$ 
& $G26_{a,b}^{b\neq 0}$, $G27_{\varepsilon,b}^{\varepsilon\in\{0,1\}}$
& None \\
\midrule

$\mathfrak{g}_2^{-1}$ 
& $G28_{a,b,c}^{bc\neq 0}$, $G29_{\varepsilon,b,c}^{\varepsilon\in\{0,1\}, c\neq 0}$, \par $G30_a$, $G31_a$, $G32_a$, $G33_{0,b}$
& $G33_{1,b}$ \\
\midrule

$\mathfrak{g}_2^{1}$ 
& $G34$--$G39$ \par (All except $G40_a^{a\neq 0}$)
& $G40_a^{a\neq 0}$ \\
\midrule

$\mathfrak{g}_2^{0}$ 
& All other structures from \par $G41$ to $G59$ not listed \par in the adjacent column
& $G45_a$, $G47_a^{a\neq 0}$, $G48_{1,b}^{b\neq 0}$, \par $G58_{\varepsilon,b}^{(\varepsilon,b)\neq(0,0)}$, $G59_{1,b}^{b\neq 0}$ \\
\midrule

$\mathfrak{sl}_2(\mathbb{C})$ 
& $G60$
& $G61$ \\

\end{longtable}

Because of space constraints, the detailed verification of the special property is presented only for the GD algebras over \(\mathfrak h\).

\begin{prop}\label{thm:special-GD-on-h3}
Let \(\mathfrak h\) be the three-dimensional Heisenberg Lie algebra over \(\mathbb C\) with basis \(\{e_1,e_2,e_3\}\) and nonzero bracket \([e_1,e_2]=e_3\). For the GD algebra structures \(G1_\alpha,\ldots,G12_\alpha\) listed above, the non-special ones are precisely \(G7,G9,G10\). All the remaining cases,
\[
G1_\alpha,\ G2_\alpha,\ G3_{\alpha,\beta}^{\alpha\neq 0},\ G4_\alpha,\ G5,\
G6_\alpha,\ G8_{\alpha,\beta},\ G11_{\alpha,\beta},\ G12_\alpha,
\]
are special.
\end{prop}

\begin{proof}
We first prove the special cases. In all constructions below we send \(e_1,e_2,e_3\) to \(x,y,z\), respectively, and use the Poisson bracket \(\{x,y\}=z\), with \(z\) Poisson central.

For \(G1_\alpha,G3_{\alpha,\beta}^{\alpha\neq0},G4_\alpha,G5,G12_\alpha\), take
\[
        P=\mathbb C[x,x^{-1},y,z]/(y^2,yz,z^2).
\]
The required Poisson derivations are:
\[
\begin{array}{c|ccc}
\text{case} & d(x) & d(y) & d(z)\\
\hline
G1_\alpha
& 0
& x^{-1}(y+\alpha z)
& x^{-1}z\\[1mm]
G3_{\alpha,\beta}^{\alpha\neq0}
& 1
& x^{-1}(\alpha y+\beta z)
& \alpha x^{-1}z\\[1mm]
G4_\alpha
& 1+x^{-1}y
& \alpha x^{-1}z
& 0\\[1mm]
G5
& 1+x^{-1}z
& 0
& 0\\[1mm]
G12_\alpha
& 1
& \alpha x^{-1}z
& 0 .
\end{array}
\]
Since \((y,z)^2=0\), the ideal \((y^2,yz,z^2)\) is stable under the bracket and under each \(d\), and a direct check on \(x,y,z\) shows that \(d\) is a Poisson derivation. The products \(u\circ v=u\,d(v)\) give exactly the five corresponding multiplication tables.

For \(G2_\alpha\), take
$
        P=\mathbb C[y,y^{-1},x,z]/(xz,z^2,x^2-yz),
$
with
$
        d(x)=y^{-1}(x+\alpha z), d(y)=1$ and $ d(z)=y^{-1}z.
$
The defining ideal is stable under the bracket and under \(d\), and \(d\) is a
Poisson derivation. The induced products are precisely those of \(G2_\alpha\).

For \(G8_{\alpha,\beta}\), take
$
        P=\mathbb C[y,y^{-1},x,z]/(xz,z^2,x^2-\alpha yz),
$
with
$
        d(x)=\beta y^{-1}z, d(y)=y^{-1}x$ and $ d(z)=0.
$
Again the defining ideal is stable and \(d\) is a Poisson derivation. The products \(u\circ v=u\,d(v)\) give \(G8_{\alpha,\beta}\).

For \(G6_\alpha\), take
$
        P=\mathbb C[x,y,z,p,q]/I_6,
$
where
$
        I_6=(xp,\ xq+\alpha z,\ yp-\alpha z,\ yq,\ zp,\ zq,\ p^2,\ pq,\ q^2),
$
let \(z,p,q\) be Poisson central, and define
$
        d(x)=p, d(y)=q, d(z)=d(p)=d(q)=0.
$
Then \(I_6\) is stable under the bracket and under \(d\), and the induced products satisfy
$
        x\circ y=-\alpha z, y\circ x=\alpha z,
$
with all other products among \(x,y,z\) equal to zero. Thus this realizes \(G6_\alpha\).

For \(G11_{\alpha,\beta}\), take
$
        P=\mathbb C[x,y,z,p,q]/I_{11},
$
where
$
        I_{11}=(xp-\alpha z,\ yp,\ xq-\beta z,\ yq-z,\ zp,\ zq,\ p^2,\ pq,\ q^2),
$
with the same Poisson bracket and the same derivation as above. Then \(I_{11}\) is stable under the bracket and under \(d\), and
$
        x\circ x=\alpha z, x\circ y=\beta z, y\circ y=z,
$
while \(y\circ x=0\). Hence \(G11_{\alpha,\beta}\) is special. This proves all the special cases.

It remains to exclude \(G7,G9,G10\). Every special GD algebra satisfies
\[
\tag{S}
[c,a\circ d]\circ b+([a,c]\circ d)\circ b
=
[c,(a\circ b)\circ d]-[c,a\circ b]\circ d.
\]
For \(G7\), substituting \(a=e_1\), \(b=e_1\), \(c=e_1\), \(d=e_2\) into \((S)\) gives \(0=-e_3\). Thus \(G7\) is not special. For \(G9\) and \(G10\), substituting \(a=e_1\), \(b=e_2\), \(c=e_2\), \(d=e_2\) gives \(e_3=0\). Hence \(G9\) and \(G10\) are not special. 
\end{proof}

Furthermore, based on the aforementioned 2,3-dimensional classification results, we consider the generality of the construction introduced in Proposition \ref{prop:n_GD_construction}. Below, we detail the specific procedure for constructing 3-dimensional GD algebras from 2-dimensional ones, and analyze the generality of this approach.

Let \(A=\operatorname{span}_{\mathbb{C}}\{x_1,x_2\}\) be a 2-dimensional GD algebra from the classification, and set \(\bar A=A\oplus \mathbb{C}h\). We fix the basis \(\{x_1,x_2,h\}\) of \(\bar A\). The operations on \(\operatorname{span}\{x_1,x_2\}\) are those of the 2-dimensional source algebra, while the extension is given by
\[
        [x_i,h]=0,\quad
        x_i\circ h=0,\quad
        h\circ h=0,\quad
        h\circ x_i=d(x_i),
        \quad \text{where }i=1,2,
\]
\(d:A\to A\) satisfies the conditions in Proposition~\ref{prop:n_GD_construction}.

The resulting 3-dimensional GD algebras are listed in Table~\ref{tab:constructed}.

\begin{longtable}{>{\raggedright\arraybackslash}p{0.28\textwidth}
                  >{\raggedright\arraybackslash}p{0.66\textwidth}}
\caption{3-dimensional GD algebras obtained from 2-dimensional GD algebras.}
\label{tab:constructed}\\
\toprule
Representative & Nonzero products and brackets \\
\midrule
\endfirsthead
\toprule
Representative & Nonzero products and brackets \\
\midrule
\endhead
\bottomrule
\endfoot

\(Z_0\)
& none. \\[0.35em]

\(Z_{\mathrm{nil}}\)
& \(h\circ x_2=x_1\). \\[0.35em]

\(Z_{\mathrm J}\)
& \(h\circ x_1=x_1,\quad h\circ x_2=x_1+x_2\). \\[0.35em]

\(Z_\mu\), \(\mu\in\mathbb C,\ \mu\sim\mu^{-1}\)
& \(h\circ x_1=x_1,\quad h\circ x_2=\mu x_2\). \\[0.35em]

\(T_1^{(0)}\)
& \([x_1,x_2]=x_1\). \\[0.35em]

\(T_1^{(1)}\)
& \([x_1,x_2]=x_1,\quad h\circ x_1=x_1\). \\[0.35em]

\(T_1^{(2)}\)
& \([x_1,x_2]=x_1,\quad h\circ x_2=x_1\). \\[0.35em]

\(T_2^{(0)}\)
& \(x_2\circ x_2=x_1\). \\[0.35em]

\(T_2^{(1)}\)
& \(x_2\circ x_2=x_1,\quad h\circ x_1=2x_1,\quad h\circ x_2=x_2\). \\[0.35em]

\(T_2^{\mathrm{Lie}}\)
& \(x_2\circ x_2=x_1,\quad [x_1,x_2]=x_1\). \\[0.35em]

\(T_3(r)\), \(r\in\mathbb C^*\)
& \(x_2\circ x_1=-x_1,\quad [x_1,x_2]=r x_1\). \\[0.35em]

\(N_1\)
& \(x_1\circ x_1=x_1,\quad x_2\circ x_2=x_2\). \\[0.35em]

\(N_2^{(0)}\)
& \(x_2\circ x_2=x_2\). \\[0.35em]

\(N_2^{(1)}\)
& \(x_2\circ x_2=x_2,\quad h\circ x_1=x_1\). \\[0.35em]

\(N_3(r)\), \(r\in\mathbb C\)
& \(x_1\circ x_2=x_1,\quad x_2\circ x_1=x_1,\quad
   x_2\circ x_2=x_2,\quad [x_1,x_2]=r x_1\). \\[0.35em]

\(N_4^{(0)}(r)\), \(r\in\mathbb C\)
& \(x_1\circ x_2=x_1,\quad x_2\circ x_2=x_2,\quad
   [x_1,x_2]=r x_1\). \\[0.35em]

\(N_4^{(1)}(r)\), \(r\in\mathbb C^*\)
& \(x_1\circ x_2=x_1,\quad x_2\circ x_2=x_2,\quad
   [x_1,x_2]=r x_1,\quad h\circ x_2=x_1\). \\[0.35em]

\(N_5^{(0)}(r)\), \(r\in\mathbb C\)
& \(x_1\circ x_2=x_1,\quad x_2\circ x_2=x_1+x_2,\quad
   [x_1,x_2]=r x_1\). \\[0.35em]

\(N_5^{(1)}(r)\), \(r\in\mathbb C^*\)
& \(x_1\circ x_2=x_1,\quad x_2\circ x_2=x_1+x_2,\quad
   [x_1,x_2]=r x_1,\quad h\circ x_2=x_1\). \\[0.35em]

\(N_6(\ell,r)\), \(\ell\in\mathbb C\setminus\{0,1\},\ r\in\mathbb C\)
& \(x_1\circ x_2=x_1,\quad x_2\circ x_1=\ell x_1,\quad
   x_2\circ x_2=x_2,\quad [x_1,x_2]=r x_1\). \\
\end{longtable}
\begin{prop}\label{thm:constructed-3d-gd}
Up to isomorphism, the 3-dimensional GD algebras obtained from Proposition~\ref{prop:n_GD_construction} are precisely those listed in Table~\ref{tab:constructed}. Apart from the identification \(Z_\mu\simeq Z_{\mu^{-1}}\) in the family \(Z_\mu\), no further identifications occur.
\end{prop}

Among the algebras in Table~\ref{tab:constructed}, the ones with nonzero Lie bracket correspond to the normal forms in the 3-dimensional classification as follows.

\renewcommand{\arraystretch}{1.15}
\begin{longtable}{>{\raggedright\arraybackslash}p{0.38\textwidth}
                  >{\raggedright\arraybackslash}p{0.56\textwidth}}
\caption{Correspondence between the constructed algebras with nonzero Lie bracket and the normal forms \(G_i\).}
\label{tab:constructed-to-G}\\
\toprule
\textbf{Isomorphism} & \textbf{Explicit isomorphism \(\Phi\)} \\
\midrule
\endfirsthead
\toprule
\textbf{Isomorphism} & \textbf{Explicit isomorphism \(\Phi\)} \\
\midrule
\endhead
\bottomrule
\endfoot

\(T_1^{(0)} \cong G55_{0,0,0}\)
& \(\Phi(x_1)=E_1,\quad \Phi(x_2)=E_3,\quad \Phi(h)=E_2\). \\

\(T_1^{(1)} \cong G49\)
& \(\Phi(x_1)=E_1,\quad \Phi(x_2)=E_3,\quad \Phi(h)=E_2\). \\

\(T_1^{(2)} \cong G58_{0,0}\)
& \(\Phi(x_1)=E_1,\quad \Phi(x_2)=E_3,\quad \Phi(h)=E_2\). \\

\(T_2^{\mathrm{Lie}} \cong G55_{1,0,0}\)
& \(\Phi(x_1)=E_1,\quad \Phi(x_2)=E_3,\quad \Phi(h)=E_2\). \\

\(T_3(r) \cong G53_{0,0,-1/r}\), \(r\in\mathbb{C}^*\)
& \(\Phi(x_1)=E_1,\quad \Phi(x_2)=rE_3,\quad \Phi(h)=E_2\). \\

\(N_3(r) \cong G56_{1/r,\,1/r}\), \(r\in\mathbb{C}^*\)
& \(\Phi(x_1)=E_1,\quad \Phi(x_2)=rE_3,\quad \Phi(h)=E_2\). \\

\(N_4^{(0)}(r) \cong G57_{0,\,1/r}\), \(r\in\mathbb{C}^*\)
& \(\Phi(x_1)=E_1,\quad \Phi(x_2)=rE_3,\quad \Phi(h)=E_2\). \\

\(N_4^{(1)}(r) \cong G58_{0,\,1/r}\), \(r\in\mathbb{C}^*\)
& \(\Phi(x_1)=rE_1,\quad \Phi(x_2)=rE_3,\quad \Phi(h)=E_2\). \\

\(N_5^{(0)}(r) \cong G57_{1,\,1/r}\), \(r\in\mathbb{C}^*\)
& \(\Phi(x_1)=r^2E_1,\quad \Phi(x_2)=rE_3,\quad \Phi(h)=E_2\). \\

\(N_5^{(1)}(r) \cong G58_{1,\,1/r}\), \(r\in\mathbb{C}^*\)
& \(\Phi(x_1)=r^3E_1,\quad
  \Phi(x_2)=r^2E_2+rE_3,\quad
  \Phi(h)=r^2E_2\). \\

\(N_6(\ell,r) \cong G56_{1/r,\,\ell/r}\),
\(\ell\in\mathbb{C}\setminus\{0,1\}\), \(r\in\mathbb{C}^*\)
& \(\Phi(x_1)=E_1,\quad \Phi(x_2)=rE_3,\quad \Phi(h)=E_2\). \\

\end{longtable}

\begin{rem}
    
It is evident that the 3-dimensional GD algebras obtained via this construction are GD algebras on the Lie algebra $\mathfrak{g}_2^0$. Consequently, all the resulting algebras are special, with the exception of $N_4^{(1)}(r)$ and $N_5^{(1)}(r)$.
\end{rem}

In the following, we provide a characterization of nilpotency and solvability for the 2,3-dimensional classification results obtained above.

\begin{prop}\label{prop:nil-solv-2d-GD}
For the 2-dimensional complex GD algebras listed in Theorem~\ref{thm:2dim-GD-structure}, the nilpotent cases are 
$
        T1(0)$ and $ T2(0).
$
The solvable cases are 
\[
        T1(0),\quad T1(1),\quad T2(0),\quad T2(1),\quad T3_r.
\]

\end{prop}

\begin{proof}
The algebras \(T1(0)\) and \(T2(0)\) are nilpotent. Indeed, \(T1(0)\) has both operations zero, while for \(T2(0)\), we have
$
        A^{\langle 2\rangle}=\mathbb C e_1,
        A^{\langle 3\rangle}=0.
$

For \(T1(1)\) and \(T2(1)\), we have
$
        A^{(1)}=\mathbb C e_1$ and $
        A^{(2)}=0,
$
so both algebras are solvable. They are not nilpotent, since their underlying Lie algebra is the non-abelian 2-dimensional Lie algebra, which is not nilpotent.

For \(T3_r\), we also have
$
        A^{(1)}=\mathbb C e_1,
        A^{(2)}=0.
$
Thus \(T3_r\) is solvable for every \(r\). It is not nilpotent: if \(r\neq0\), the underlying Lie algebra is not nilpotent. If \(r=0\), then \(e_2\circ e_1=-e_1\), so
$
        A^{\langle n\rangle}=\mathbb C e_1,
      \text{for } n\geq2.
$

It remains to exclude the \(N\)-type families. For
$
        N1,N3_r,N4_r,N5_r, N6_{\ell,r},
$
the displayed Novikov products give \(A\circ A=A\). Hence \(A^{(1)}=A\), so these algebras are not solvable. For \(N2\), we have
$
        A^{(1)}=\mathbb C e_2,
        A^{(2)}=\mathbb C e_2,
$
because \(e_2\circ e_2=e_2\). Thus \(N2\) is not solvable.

Therefore the solvable cases are
$
        T1(0),T1(1), T2(0),T2(1),T3_r,
$
and among them the nilpotent cases are \(T1(0)\) and \(T2(0)\).
\end{proof}

\begin{prop}\label{thm:nil-solv-G1-G61}
For the complex 3-dimensional GD algebras \(G1,\ldots,G61\) listed above, the nilpotent cases are 
$
        G6_\alpha,
        G8_{\alpha,\beta}$ and $
        G11_{\alpha,\beta}.
$
The solvable cases are 
\[
\begin{gathered}
        G1_\alpha,\quad
        G6_\alpha,\quad
        G8_{\alpha,\beta},\quad
        G11_{\alpha,\beta},                               
        G13_{0,\beta,\gamma}^{\beta\neq0},\quad
        G14_{0,\beta},\quad
        G15_{0,\beta},\quad
        G16_{0,\beta},\quad
        G17_0, \quad 
         G18_{\alpha,0,\gamma,\delta}
        \\[1mm]
        G19_{\alpha,\beta,0,\delta},\quad
        G20_{\alpha,\beta,0,\delta},\quad
        G21_{\alpha,\beta,\gamma,0},                      
        G22_{\alpha,\beta,0},\quad
        G23_{\alpha,\beta,0},\quad
        G24_{0,\beta},\quad
        G25_{\alpha,0},\quad
        G26_{0,\beta}                                       \\[1mm]
         G27_{\alpha,0}, \quad   
        G28_{0,\beta,\gamma},\quad
        G29_{\alpha,0,\gamma},\quad
        G30_0,\quad
        G31_0,\quad
        G32_0,\quad
        G33_{\alpha,0},                                   
        G34_{0,\beta,\gamma},\quad
        G35_{\alpha,0,\gamma},                                             \\[1mm]
        G36_0,\quad
        G37_0,\quad
        G38_0, \quad 
        G44,\quad
        G45_\alpha,\quad
        G46,\quad
        G49,                                               
        G52_{0,\beta,\gamma}^{\beta\gamma\neq0},\quad
        G53_{\alpha,0,\gamma}^{\alpha\in\{0,1\},\,\gamma\neq0},
        \\[1mm]
         G54_{\alpha,0,\gamma}^{\alpha\in\{0,1\},\,\gamma\neq0},  \quad    
        G55_{\alpha,\beta,0}^{\alpha,\beta\in\{0,1\}},\quad
        G58_{\alpha,0}^{\alpha\in\{0,1\}} .
\end{gathered}
\]

\end{prop}

\begin{proof}
We first determine the nilpotent cases. By Theorem~\ref{thm:main}, a nilpotent GD algebra must have nilpotent underlying Lie algebra. Hence all families whose underlying Lie algebra is not nilpotent are excluded. This removes \(G13,\ldots,G61\): in \(G13,\ldots,G40\) the lower central series of the underlying Lie algebra stabilizes at \(\operatorname{span}\{e_1,e_2\}\), in \(G41,\ldots,G59\) it stabilizes at \(\mathbb C e_1\), and in \(G60,G61\) the underlying Lie algebra is simple.

Thus only \(G1,\ldots,G12\) can be nilpotent. Among them, direct computation from \eqref{eq:GD-powers} gives
$
        A^{\langle3\rangle}=0
$
for \(G6_\alpha\) and \(G11_{\alpha,\beta}\), since all products and brackets take values in \(\mathbb C e_3\) and \(e_3\) annihilates the algebra. For \(G8_{\alpha,\beta}\), we have
\[
        A^{\langle2\rangle}=\operatorname{span}\{e_1,e_3\},\qquad
        A^{\langle3\rangle}=\mathbb C e_3,\qquad
        A^{\langle4\rangle}=0.
\]
Hence these three families are nilpotent. The remaining Heisenberg families are not nilpotent: \(G1_\alpha\) satisfies
$
        A^{\langle n\rangle}=\operatorname{span}\{e_2,e_3\}
        $ for $ n\geq2,
$
while all other non-solvable cases listed below are automatically non-nilpotent.

We now determine solvability by computing the first terms of the GD derived series. For \(G1,\ldots,G12\), one obtains
$
        A^{(2)}=0
$
for
\[
        G1_\alpha,\quad G6_\alpha,\quad
        G8_{\alpha,\beta},\quad G11_{\alpha,\beta}.
\]
All other families in this block have either \(A^{(1)}=A\), or a nonzero stable derived subspace, hence they are not solvable.

For \(G13,\ldots,G38\), we have
$
        [A,A]=\operatorname{span}\{e_1,e_2\}.
$
Moreover, all products except possibly \(e_3\circ e_3\) take values in \(\operatorname{span}\{e_1,e_2\}\). Therefore, if the \(e_3\)-component of \(e_3\circ e_3\) is nonzero, then \(A^{(1)}=A\), and the algebra is not solvable. If this component is zero, then \(A^{(1)}\subseteq \operatorname{span}\{e_1,e_2\}\), and the induced products on this subspace are either zero or one of the nilpotent products
$
        e_1\circ e_1=e_2$ and $ e_2\circ e_2=e_1.
$
Thus the GD derived series terminates exactly in the parameter cases listed above from \(G13\) to \(G38\).

For \(G39,\ldots,G59\), the direct GD derived computations give
$
        A^{(2)}=0
$
for
\[
        G44,\quad G45_\alpha,\quad G46,\quad G49.
\]

For \(G52,G53,G54,G55\) and \(G58\), solvability is again equivalent to the vanishing of the \(e_3\) component of \(e_3\circ e_3\), giving 
\[
        G52_{0,\beta,\gamma},\quad
        G53_{\alpha,0,\gamma},\quad
        G54_{\alpha,0,\gamma},\quad
        G55_{\alpha,\beta,0},\quad
        G58_{\alpha,0}.
\]
The remaining families in this block either have \(A^{(1)}=A\), or contain a nonzero stable derived subspace, for instance the one generated by \(e_2\circ e_2=e_2\) in \(G50\) and \(G51\). Hence they are not solvable.

Finally, \(G60\) and \(G61\) have simple underlying Lie algebra. Therefore \([A,A]=A\), so \(A^{(1)}=A\). They are not solvable.
This completes the proof.
\end{proof}

\section*{Acknowledgements}
Finally, the authors would like to thank Professor Ivan Kaygorodov for his valuable suggestions on the content of this paper.

\end{document}